\documentclass[12pt]{amsart}
\usepackage{amssymb,amsfonts,amsthm,latexsym,bm}
\usepackage{hyperref}
\usepackage[centertags]{amsmath}
\usepackage{mathrsfs}
\usepackage{t1enc}
\usepackage{graphicx}
\usepackage{enumerate}

\advance\headheight by 3pt

\newtheorem{theorem}{Theorem}[section]

\newtheorem{lemma}[theorem]{Lemma}

\newtheorem*{theorem*}{Theorem}
\newtheorem{proposition}[theorem]{Proposition}

\newtheorem{thmalpha}{Theorem}
  
\numberwithin{equation}{section}
\def\pitem{\advance\leftskip3mm\advance\linewidth-3mm}
\def\mitem{\advance\leftskip-3mm\advance\linewidth3mm}

\begingroup\makeatletter\ifx\SetFigFont\undefined
\gdef\SetFigFont#1#2#3#4#5{
  \reset@font\fontsize{#1}{#2pt}
  \fontfamily{#3}\fontseries{#4}\fontshape{#5}
  \selectfont}
\fi\endgroup

\theoremstyle{definition}

\renewcommand{\subjclass}[1]{\thanks{\emph{2020 Mathematics Subject Classification:}~#1}}
\renewcommand{\keywords}[1]{\thanks{\emph{Keywords and Phrases:}~#1}}
\renewcommand{\date}{\thanks{\today}}

\newcommand{\CC}{\mathcal{C}}
\newcommand{\DD}{\mathcal{D}}

\newcommand{\FF}{\mathcal{F}}

\newcommand{\HH}{\mathcal{H}}
\newcommand{\II}{\mathcal{I}}

\newcommand{\MM}{\mathcal{M}}
\newcommand{\NN}{\mathcal{N}}
\newcommand{\OO}{\mathcal{O}}
\newcommand{\PP}{\mathcal{P}}

\renewcommand{\SS}{\mathcal{S}}
\newcommand{\TT}{\mathcal{T}}

\newcommand{\VV}{\mathcal{V}}

\newcommand{\fp}{\mathfrak{p}}

\newcommand{\fa}{\mathfrak{a}}

\newcommand{\fd}{\mathfrak{d}}

\newcommand{\fP}{\mathfrak{P}}

\newcommand{\Qq}{\mathbb{Q}}

\newcommand{\Zz}{\mathbb{Z}}

\newcommand{\Pp}{\mathbb{P}}

\newcommand{\vk}{\Bbbk}

\newcommand{\ve}{\varepsilon}

\def\house#1{\setbox1=\hbox{$\,#1\,$}%
\dimen1=\ht1 \advance\dimen1 by 2pt \dimen2=\dp1 \advance\dimen2
by 2pt
\setbox1=\hbox{\vrule height\dimen1 depth\dimen2\box1\vrule}%
\setbox1=\vbox{\hrule\box1}%
\advance\dimen1 by .4pt \ht1=\dimen1 \advance\dimen2 by .4pt
\dp1=\dimen2 \box1\relax}

\newcommand{\kdots}{,\ldots ,}

\newcommand{\GL}{{\rm GL}}

\newcommand{\half}{\mbox{$\textstyle{\frac{1}{2}}$}}

\newcommand{\smallfrac}[2]{\mbox{$\textstyle{\frac{#1}{#2}}$}}

\newcommand{\medfrac}[2]{\mbox{\large{$\textstyle{\frac{#1}{#2}}$}}}

\renewcommand{\mod}[3]{#1\equiv#2\,({\rm mod}\,#3)}
\renewcommand{\gcd}{{\rm gcd}}

\newcommand{\cross}{{\rm cr}}

\title{Orders with few rational monogenizations}
\author[J.-H. Evertse]{Jan-Hendrik Evertse}
\address{J.-H. Evertse \newline
         \indent Universiteit Leiden, Mathematisch Instituut, \newline
         \indent Postbus 9512, 2300 RA Leiden, The Netherlands \newline
         \indent \textit{URL:} {\tt https://pub.math.leidenuniv.nl/$\sim$evertsejh}}
\email{evertse\char'100math.leidenuniv.nl}

\begin{document}

\begin{abstract}
Recall that a monogenic order is an order of the shape $\Zz [\alpha ]$, where $\alpha$ is an algebraic integer. This is generalized to orders $\Zz_{\alpha}$ for not necessarily integral algebraic numbers $\alpha$ as follows.
For  an algebraic number $\alpha$ of degree $n$,
let $\MM_{\alpha}$ be the $\Zz$-module
generated by $1,\alpha\kdots\alpha^{n-1}$; then $\Zz_{\alpha}:=\{\xi\in\Qq (\alpha ):\, \xi\MM_{\alpha}\subseteq\MM_{\alpha}\}$ is the ring of scalars of $\MM_{\alpha}$.
We call an order of the shape $\Zz_{\alpha}$ \emph{rationally monogenic}.
If $\alpha$ is an algebraic integer, then $\Zz_{\alpha}=\Zz [\alpha ]$ is monogenic.
In fact, rationally monogenic orders are 
special cases of invariant rings of polynomials or binary forms, which were introduced 
by Birch and Merriman (1972), Nakagawa (1989), and Simon (2001).
If $\alpha ,\beta$ are two $\GL_2(\Zz )$-equivalent algebraic numbers, i.e., $\beta =\frac{a\alpha +b}{c\alpha +d}$ for some $\big(\begin{smallmatrix}a&b\\c&d\end{smallmatrix}\big)\in\GL_2(\Zz )$, then $\Zz_{\alpha}=\Zz_{\beta}$. Given an order $\OO$ of a number field,
we call a $\GL_2(\Zz )$-equivalence class of $\alpha$ with $\Zz_{\alpha}=\OO$ a \emph{rational monogenization}
of $\OO$.

We prove the following. If $K$ is a quartic number field, 
then $K$ has only finitely many orders with more than two rational monogenizations. This is best possible.
Further, if $K$ is a number field of degree $\geq 5$, the Galois group of whose normal closure is 
$5$-transitive, then $K$ has only finitely many orders with more than  one rational monogenization. The proof uses finiteness results for unit equations, which in turn were derived
from Schmidt's Subspace Theorem. 

We generalize the above results to rationally monogenic orders
over rings of $S$-integers
of number fields.

Our results extend work of B\'{e}rczes, Gy\H{o}ry and the author from 2013
on multiply monogenic orders.
\end{abstract}
 
\maketitle

\section{Introduction}\label{section1}

\textbf{Summary.} Recall that a monogenic order is an order of the shape $\Zz [\alpha ]$, where $\alpha$ is an algebraic integer. This is generalized to orders $\Zz_{\alpha}$ for not necessarily integral algebraic numbers $\alpha$ as follows. For an algebraic number $\alpha$ of degree $n$,
let $\MM_{\alpha}$ be the $\Zz$-module generated by $1,\alpha\kdots\alpha^{n-1}$; 
then $\Zz_{\alpha}:=\{\xi\in\Qq (\alpha ):\, \xi\MM_{\alpha}\subseteq\MM_{\alpha}\}$ is the ring of scalars of $\MM_{\alpha}$.
We call an order of the shape $\Zz_{\alpha}$ \emph{rationally monogenic}.
If $\alpha$ is an algebraic integer, then $\Zz_{\alpha}=\Zz [\alpha ]$ is monogenic.
Rationally monogenic orders are invariant rings of primitive polynomials or binary forms,
see, e.g., \cite{BM72}, \cite{N89}, \cite{S01}, \cite{S03}, \cite{DDS05}, \cite{W11}, \cite[Chap. 16]{EG17}.
If $\alpha ,\beta$ are two $\GL_2(\Zz )$-equivalent algebraic numbers, i.e., $\beta =\frac{a\alpha +b}{c\alpha +d}$ for some $\big(\begin{smallmatrix}a&b\\c&d\end{smallmatrix}\big)\in\GL_2(\Zz )$, then 
$\Zz_{\alpha}=\Zz_{\beta}$. Given an order $\OO$ of a number field,
we call a $\GL_2(\Zz )$-equivalence class of $\alpha$ with $\Zz_{\alpha}=\OO$ a \emph{rational monogenization}
of $\OO$.

We prove the following. If $K$ is a quartic number field, 
then $K$ has only finitely many orders with more than two rational monogenizations. This is best possible.
Further, if $K$ is a number field of degree $\geq 5$, the Galois group of whose normal closure is 
$5$-transitive, then $K$ has only finitely many orders with more than  one rational monogenization. The proof uses finiteness results for unit equations, which in turn were derived
from Schmidt's Subspace Theorem. 
Except for the hypothesis on the normal closure of $K$, our result implies a conjecture posed in \cite{BEGRS22}.

We generalize the above results to rationally monogenic orders
over rings of $S$-integers
of number fields. Our results extend work of B\'{e}rczes, Gy\H{o}ry and the author \cite{BEG13}
on monogenic orders.
\vskip0.2cm

\textbf{Background and results.} Let $K$ be a number field. Denote its ring of integers by $\OO_K$.
An order $\OO$ of $K$ (i.e., a subring of $K$ that as a $\Zz$-module is free of rank $[K:\Qq ]$)
is called \emph{monogenic} if there is $\alpha\in\OO$ with $\OO =\Zz [\alpha ]$. 
The set of $\alpha$ with $\Zz [\alpha ]=\OO$ can be divided into so-called $\Zz$-equivalence classes,
where $\alpha_1,\alpha_2$ are called $\Zz$-equivalent if $\alpha_1-\alpha_2\in\Zz$ or $\alpha_1+\alpha_2\in\Zz$. A $\Zz$-equivalence class of $\alpha$ with $\Zz [\alpha ]=\OO$ is called a
\emph{monogenization} of $\OO$.
Every order of a quadratic number field has precisely one monogenization.
Orders of number fields of degree $\geq 3$ may be non-monogenic or have more than one monogenization.
From work of Gy\H{o}ry \cite{G73}, \cite{G76} it can be deduced, and in fact in an effective
form, that if $K$ is any number field of degree $\geq 3$ 
then every order $\OO$ of $K$ has at most finitely many monogenizations.
If one keeps the number field $K$ fixed and restricts to monogenic orders of  $K$, then most of these
have only few monogenizations. 
B\'{e}rczes, Gy\H{o}ry and the author 
\cite[Theorem 1.1]{BEG13} obtained the following result.

\begin{thmalpha}\label{thmA}
Let $K$ be a number field of degree $\geq 3$. Then $K$ has only finitely many orders
with more than two monogenizations.
\end{thmalpha}

This result is optimal. For instance, if $\varepsilon$ is a unit of $\OO_K$ with $\Qq (\varepsilon )=K$, then $\Zz [\varepsilon ]=\Zz [\varepsilon^{-1}]$, while $\varepsilon$ and $\varepsilon^{-1}$
are not $\Zz$-equivalent.  More generally, let $\alpha \in \OO_K$ be such that $\Qq (\alpha )=K$, 
suppose there are integers $c,d$ such that $c\alpha +d$ is a unit of $\OO_K$, let $a,b$ be integers such that  $\big(\begin{smallmatrix}a&b\\c&d\end{smallmatrix}\big)\in\GL_2(\Zz )$, and put 
$\beta := \medfrac{a\alpha +b}{c\alpha +d}$. Then $\Zz [\alpha ]=\Zz [\beta ]$,
while $\alpha$ and $\beta$ are not $\Zz$-equivalent.

This suggests that it is natural to consider $\GL_2(\Zz )$-equivalence classes of $\alpha$ with $\Zz [\alpha ]=\OO$.
Here, $\alpha,\beta\in K$ are called $\GL_2(\Zz )$-equivalent if there is
$\big(\begin{smallmatrix}a&b\\c&d\end{smallmatrix}\big)\in\GL_2(\Zz )$ such that $\beta =\medfrac{a\alpha +b}{c\alpha +d}$.

We say that a group $G$ acts $t$-transitively on a finite set $\SS$
if for any pairwise distinct $i_1\kdots i_t\in\SS$ and pairwise distinct $j_1\kdots j_t\in\SS$, there is $\sigma\in G$
such that $\sigma (i_1)=j_1\kdots \sigma (i_t)=j_t$.
If $K =\Qq (\alpha )$ and $L$ is the normal closure of $K$, we say that ${\rm Gal}(L/\Qq )$
is $t$-transitive if it acts $t$-transitively on the set of conjugates
$\{ \alpha^{(1)}\kdots \alpha^{(n)}\}$ of $\alpha$.

Then from \cite[Theorems 1.1 and 1.2(ii)]{BEG13}, the following can be deduced:

\begin{thmalpha}\label{thmB}
Let $K$ be a number field of degree $\geq 5$ such that the Galois group of its normal closure is 
$4$-transitive. Then for all orders $\OO$ of $K$ with at most finitely many exceptions, the set of $\alpha$ with $\Zz [\alpha ]=\OO$ is contained in at most one $\GL_2(\Zz )$-equivalence class.
\end{thmalpha}

It is not known whether the condition on the normal closure of $K$ is necessary. It can be proved in an elementary way that if $K$ is a cubic number field and $\OO$ an order of $K$, then the set of $\alpha\in\OO$ with $\Zz [\alpha ]=\OO$ is contained in at most one $\GL_2(\Zz )$-equivalence class. For quartic number fields $K$, the above theorem is false. In fact, \cite[end of Section 1]{BEG13} gives the following
construction:

\begin{thmalpha}\label{thmC}
Let $r,s$ be integers such that $f(X)=(X^2-r)^2-X-s$ is irreducible, and let $K=\Qq (\alpha )$,
where $\alpha$ is a root of $f$. Then $K$ has infinitely many orders $\OO_m$ ($m=1,2,\ldots$)
with the following property: $\OO_m=\Zz [\alpha_m]=\Zz [\beta_m]$, where $\beta_m=\alpha_m^2-r_m$,
$\alpha_m =\beta_m^2-s_m$ for some integers $r_m,s_m$.
\end{thmalpha}
\noindent
It is clear that $\alpha_m,\beta_m$ in the above theorem are not $\GL_2(\Zz )$-equivalent.

Our aim is to generalize Theorem \ref{thmB} to orders attached to non-integral algebraic numbers.
Let $\alpha$ be an algebraic number of degree $n$ and 
$f_{\alpha}\in\Zz [X]$ its primitive minimal
polynomial, i.e., with coefficients having gcd $1$. Then the order $\Zz_{\alpha}$
attached to $\alpha$ is 
the invariant ring or order of $f_{\alpha}$, see Nakagawa \cite{N89}, 
Simon \cite{S01} or \cite{BM72}, \cite{S03}, \cite{DDS05}, \cite{W11}, \cite[Chap. 16]{EG17}.
Nakagawa and Simon defined this order by giving a $\Zz$-module basis for it, together
with a multiplication table.  
A direct definition of $\Zz_{\alpha}$ is as follows. Define the $\Zz$-module
\begin{equation}\label{eq1.1}
\MM_{\alpha}:=\big\{ x_0+x_1\alpha +\cdots +x_{n-1}\alpha^{n-1}:\, x_0\kdots x_{n-1}\in\Zz\}.
\end{equation}
Then $\Zz_{\alpha}$ is the ring of scalars of $\MM_{\alpha}$, i.e.,
\begin{equation}\label{eq1.2}
\Zz_{\alpha}:=\{ \xi \in\Qq (\alpha ):\, \xi\MM_{\alpha}\subseteq\MM_{\alpha}\}.
\end{equation}
If $\alpha$ is an algebraic integer, then $\alpha^i\in\MM_{\alpha}$ for $i\geq n$,
and thus, $\Zz_{\alpha}=\MM_{\alpha}=\Zz [\alpha]$.
Further, if $\alpha ,\beta$ are $\GL_2(\Zz )$-equivalent, i.e.,
$\beta =\medfrac{a\alpha +b}{c\alpha +d}$ for some 
$\big(\begin{smallmatrix}a&b\\c&d\end{smallmatrix}\big)\in\GL_2(\Zz )$,
then one easily verifies that
$\MM_{\beta }=\\(c\alpha +d)^{1-n}\MM_{\alpha}$,
which implies $\Zz_{\beta}=\Zz_{\alpha}$. 

To simplify the formulation of our results, we introduce the following
terminology. We call an order $\OO$ of a number field $K$
\emph{rationally monogenic} if $\OO =\Zz_{\alpha}$ for some $\alpha$ with $K=\Qq (\alpha )$.
A $\GL_2(\Zz )$-equivalence class of $\alpha$ with $\Zz_{\alpha}=\OO$
is called a \emph{rational monogenization} of $\OO$. 

We give some other descriptions for $\Zz_{\alpha}$. Let again $\alpha$ be an algebraic number of degree $n$, and denote by $f_{\alpha}$
its primitive minimal polynomial, i.e., $f_{\alpha}=a_0X^n+\cdots +a_n\in\Zz [X]$ with $a_0>0$
and $\gcd (a_0\kdots a_n)=1$. Then $\Zz_{\alpha}$ is the $\Zz$-module with basis
\begin{equation}\label{eq1.3}
1,\omega_1\kdots \omega_{n-1},\ \ \omega_i=a_0\alpha^i+a_1\alpha^{i-1}+\cdots a_{i-1}\alpha\ \ 
(i=1\kdots n-1)
\end{equation}
(see \cite[p. 365, Thm. 16.2.9, formula (16.2.7)]{EG17} or Lemma \ref{lem2.0} in the present paper).
This is precisely the invariant order of $f_{\alpha}$ as defined by Nakagawa \cite{N89} and Simon
\cite{S01}. Del Corso, Dvornicich and Simon \cite[Prop. 2]{DDS05} (see also Lemma \ref{lem2.0}
in the present paper) proved the much simpler expression
\[
\Zz_{\alpha}=\Zz [\alpha ]\cap \Zz [\alpha^{-1}].
\] 
From the basis \eqref{eq1.3}
one deduces that
the discriminant
of the order $\Zz_{\alpha}$ is equal to the discriminant of $f_{\alpha}$, i.e.,
\begin{align}\label{eq1.4}
D(\Zz_{\alpha}) &=D_{\Qq (\alpha )/\Qq}(1,\omega_1\kdots\omega_{n-1})
\\
\notag
&=
a_0^{2n-2}D_{\Qq (\alpha)/\Qq}(1,\alpha\kdots \alpha^{n-1})
\\
\notag
&=
a_0^{2n-2}\prod_{1\leq i<j\leq n}(\alpha^{(i)}-\alpha^{(j)})^2
=D(f_{\alpha}),
\end{align}
where $\alpha^{(1)}\kdots\alpha^{(n)}$ are the conjugates of $\alpha$.

The orders $\Zz_{\alpha}$ are part of a much more general theory on invariant rings of binary forms, 
see \cite{N89}, \cite{S03}, \cite{DDS05}, \cite{W11}, \cite[Chap. 16]{EG17}. 
We briefly comment on this at the end of this section. 

It follows from the work of Birch and Merriman \cite{BM72} on binary forms that an order
of a number field has at most finitely many rational monogenizations.
Gy\H{o}ry and the author \cite[Cor. 2]{EG91} proved that every algebraic number $\alpha$
of degree $n$ is $\GL_2(\Zz )$-equivalent to an algebraic number $\alpha^*$ with height
$H(\alpha^*)\leq C(n,D)$, where $H(\alpha^*)$ is the maximum of the absolute values
of the coefficients of $f_{\alpha^*}$, $D$ is the discriminant of $f_{\alpha}$, and $C(n,D)$
is effectively computable. Together with \eqref{eq1.4} this implies that it can be decided 
effectively whether a given order of a number field has rational monogenizations and that these
can be determined effectively.

It can be shown that a rationally monogenic order $\OO$ of a number field
of degree $\geq 3$ is \emph{primitive}, i.e., there are no
order $\OO'$ and integer $a>1$ such that $\OO =\Zz +a\OO'$. It follows from
classical work of Delone and Faddeev \cite{DF40} that every primitive order of a cubic
number field has precisely one rational monogenization. Further, 
work of B\'{e}rczes, Gy\H{o}ry and the author \cite{BEG04} implies
that an order of a number field of degree $n\geq 4$ cannot have more than $n\cdot 2^{24n^3}$
rational monogenizations. Gy\H{o}ry and the author \cite[Chap. 17]{EG17} improved this to $2^{5n^2}$.
From recent work of Bhargava \cite{B22} it follows that for quartic orders this bound can be improved
to $40$.

We are now ready to state the main result of this paper,
which gives a generalization of Theorem \ref{thmB}
to not necessarily integral algebraic numbers $\alpha$.

\begin{theorem}\label{thm1.1}
(i) Let $K$ be a quartic number field. 
%and suppose that its normal closure has as Galois group
%the full symmetric group $S_4$. 
Then $K$ has only finitely many orders 
with more than two rational monogenizations.
\\[0.15cm]
(ii) Let $K$ be a number field of degree $\geq 5$ 
and suppose that the Galois group of its normal closure is $5$-transitive. 
Then $K$ has only finitely many orders with more than one rational monogenization.
\end{theorem}

Theorem \ref{thmC} implies that there are quartic number fields, having infinitely many 
orders with two rational monogenizations.
We do not know whether the condition on the normal closure of $K$ is necessary
if $[K:\Qq ]\geq 5$. Probably, trying to remove or relax this condition
would considerably complicate the proof.

The proof of Theorem \ref{thm1.1} uses among other things finiteness results for unit equations in more than two unknowns. The present proofs of these depend on ineffective methods 
from Diophantine approximation, e.g., Schmidt's Subspace Theorem or the Faltings-R\'{e}mond method.
As a consequence, our proof of
Theorem \ref{thm1.1} is ineffective in that it does not allow to determine
the exceptional orders. Further, although for unit equations we have good upper bounds for the 
number of solutions, it is because of the `other things,' that we cannot give an upper bound for the
number of exceptional orders.

We state a consequence, which partly confirms Conjecture 4.2 in \cite{BEGRS22}.
We adopt the terminology of \cite{BEGRS22}.
Given a number field $K$, denote by $\PP\II (K)$ the set of primitive, irreducible polynomials
$f\in\Zz [X]$, such that there is $\alpha$ with $f(\alpha )=0$ and
$\Qq (\alpha )=K$. We call two polynomials $f,g\in\PP\II (K)$ $\GL_2(\Zz )$-equivalent if 
there is $\big(\begin{smallmatrix}a&b\\c&d\end{smallmatrix}\big)\in\GL_2(\Zz )$ such that $g(X)=\pm (cX+d)^{\deg f}f\big(\frac{aX+b}{cX+d}\big)$. Further, $f$ and $g$ are called
\emph{Hermite equivalent} if there are $\alpha ,\beta$ such that $\Qq (\alpha )=\Qq (\beta ) =K$, $f(\alpha)=0$, $g(\beta)=0$ and 
$\MM_{\beta}=\lambda\MM_{\alpha}$ for some $\lambda\in K^*$ (see \eqref{eq1.1} above).
It was shown in \cite{BEGRS22} that two $\GL_2(\Zz )$-equivalent polynomials are Hermite equivalent.
As we will show, Theorem \ref{thm1.1} implies the following,
which except for the assumption on the normal closure of $K$ is Conjecture 4.2 of \cite{BEGRS22}. 

\begin{theorem}\label{thm1.2}
(i) Let $K$ be a quartic number field. 
%whose normal closure has Galois group $S_4$.
Then there are only finitely many Hermite equivalence classes in $\PP\II (K)$ that 
fall apart into more than two $\GL_2(\Zz )$-equivalence classes.
\\[0.15cm]
(ii) Let $K$ be a number field of degree $\geq 5$, 
such that the Galois group of its normal closure  is $5$-transitive.
Then there are only finitely many Hermite equivalence classes in $\PP\II (K)$ 
that fall apart into more than one $\GL_2(\Zz )$-equivalence class.
\end{theorem}

Another consequence of our investigations, which probably could be proved by other
means as well, is the following.

\begin{theorem}\label{thm1.3}
Let $K$ be a number field of degree $\geq 3$.
Then $K$ has infinitely many orders that are rationally monogenic but not monogenic.
\end{theorem}

Finally, we would like to comment on the connection between the orders $\Zz_{\alpha}$
defined above, and invariant orders of binary forms.
Birch and Merriman \cite{BM72} introduced for a binary form
\[
F(X,Y) =a_0X^n+a_1X^{n-1}Y+\cdots +a_nY^n\in\Zz [X,Y]
\]
that is irreducible over $\Qq$ the $\Zz$-module $\Zz_F$ with $\Zz$-basis $1,\omega_1\kdots\omega_{n-1}$ given by \eqref{eq1.3},
where $F(\alpha ,1)=0$.
Nakagawa \cite{N89} proved that $\Zz_F$ is an order of the number field 
$\Qq (\alpha )$, in fact,
\begin{equation}\label{eq1.6}
\omega_i\omega_j=-\sum_{\max (i+j-n,1)\leq k\leq i} a_{i+j-k}\omega_k+
\sum_{j<k\leq \min (i+j,n)} a_{i+j-k}\omega_k
\end{equation}
for $i,j=1\kdots n-1$, where $\omega_n:=-a_n$.
Thus, $\Zz_F$ is called the invariant ring or order of $F$.
This order was further studied by Simon \cite{S01,S03} and
Del Corso, Dvornicich and Simon \cite{DDS05}.

Notice that in the definition of $\Zz_F$ we did not require that the coefficients of $F$ have greatest common divisor $1$.
Our order $\Zz_{\alpha}$ is just $\Zz_F$ where $F(X,Y)=Y^{\deg \alpha}f_{\alpha}(X/Y)$
is an irreducible binary form
whose coefficients have greatest common divisor $1$.

More generally, 
given any commutative ring $R$ and binary form $F=\sum_{i=0}^n a_iX^{n-i}Y^i \in R[X,Y]$,
one can formally define the invariant ring $R_F$ of $F$ by taking the free $R$-module
with basis $1,\omega_1\kdots \omega_{n-1}$ with prescribed multiplication table \eqref{eq1.6}.
Here, it is no longer required
that $F$ is irreducible, nor even that $a_0\not=0$, and even $a_0=\cdots =a_n=0$ is allowed.
Wood \cite{W11} studied invariant rings of binary forms in a
much broader context.

The remainder of our paper is organized as follows. 
In Section \ref{section2} we have collected some basic properties
of rationally monogenic orders. Although these are all known, we have provided
proofs for convenience of the reader.
Sections \ref{section3} and \ref{section4} contain preparations,
where in Section \ref{section3} we apply finiteness results for unit equations.
In Section \ref{section5} we finish the proofs of Theorems \ref{thm1.1}--\ref{thm1.3}.
Finally, in Section \ref{section6} 
we generalize the orders $\Zz_{\alpha}$ to domains $\OO_{S,{\alpha}}$,
where $\OO_S$ is the ring of $S$-integers of a number field $\vk$
and $\alpha$ is algebraic over $\vk$,
and state and prove
a generalization of Theorem \ref{thm1.1} but with a notion of equivalence that is slightly weaker
than $\GL_2(\OO_S)$-equivalence.

\section{Lemmas over principal ideal domains}\label{section2}

In this section, we have collected some generalities on rationally monogenic orders. We state and prove
everything over an arbitrary principal ideal domain $A$ of characteristic $0$.
Most of the results in this section have been proved elsewhere in a more general context,
see for instance \cite[Chaps. 16, 17]{EG17}, \cite{BEG04}, \cite{DDS05}.
For convenience of the reader we have repeated the short proofs, specialized to the
situation of this paper.
In the proofs of Theorems \ref{thm1.1}--\ref{thm1.3} we apply the results of the present section with $A=\Zz$. In Section \ref{section6}
we use a local-to-global argument, and apply the results of the present section
to localizations of $\OO_S$.

In what follows, if $F$ is any field, $\xi\in \Pp^1(F):=F\cup \{ \infty\}$ and 
$C=\big(\begin{smallmatrix}a&b\\c&d\end{smallmatrix}\big)\in\GL_2(F)$, we write $C\xi :=\frac{a\xi +b}{c\xi +d}$, with the conventions that this
is $\infty$ if $\xi =\infty$ and $c=0$; $a/c$ if $\xi =\infty$ and $c\not= 0$;
$\infty$ if $c\not= 0$ and $\xi = -d/c$.

Let $A$ be a principal ideal domain of characteristic $0$, and $\vk$ its field of fractions.
Fix a finite extension $K$ of $\vk$ of degree $n\geq 3$. Let $L$ be its normal
closure over $\vk$ and $x\mapsto x^{(i)}$ ($i=1\kdots n$) the $\vk$-isomorphic embeddings
of $K$ in $L$. Further, denote by $A_K$, $A_L$ the integral closures of $A$ in $K$ and $L$,
respectively. Recall that both $A_K$, $A_L$ are Dedekind domains; in the case that $A=\Zz$,
$A_K$ and $A_L$ are just the rings of integers of $K$ and $L$.

Given any domain $B\supseteq A$, we call $\alpha ,\beta\in K$ $\GL_2(B)$-equivalent 
if there is $C\in\GL_2(B)$ 
such that $\beta=C\alpha$.

Let $\alpha\in K$ with $K=\vk (\alpha )$. Define the free $A$-module
\begin{equation}\label{eq2.1}
\MM_{\alpha}:=\big\{ x_0+x_1\alpha +\cdots +x_{n-1}\alpha^{n-1}:\, x_0\kdots x_{n-1}\in A\}
\end{equation}
and its ring of scalars
\begin{equation}\label{eq2.2}
A_{\alpha}:=\{ \xi\in K:\, \xi\MM_{\alpha}\subseteq\MM_{\alpha}\}.
\end{equation}
As one easily verifies, if $\alpha ,\beta$ are two $\GL_2(A)$-equivalent elements of $A$,
then $\MM_{\alpha}=\lambda\MM_{\beta}$ for some $\lambda\in K^*$, and thus, $A_{\alpha}=A_{\beta}$.

We give some other descriptions of $A_{\alpha}$.
Let $f_{\alpha}=a_0X^n+\cdots +a_n\in A[X]$
be a primitive minimal polynomial of $\alpha$, i.e., with $\gcd (a_0\kdots a_n)=1$.
Such a polynomial exists since $A$ is a principal ideal domain. 

\begin{lemma}\label{lem2.0}
We have
\begin{equation}\label{eq2.3}
A_{\alpha}=\big\{ x_0+x_1\omega_1+\cdots +x_{n-1}\omega_{n-1}:\, x_0\kdots x_{n-1}\in A\big\} 
\end{equation}
where 
\[
\omega_i:=a_0\alpha^i+a_1\alpha^{i-1}+\cdots +a_{i-1}\alpha\ \ (i=1\kdots n-1),
\]
and
\begin{equation}\label{eq2.5}
A_{\alpha}=A[\alpha ]\cap A[\alpha^{-1}].
\end{equation}
\end{lemma}
Identity \eqref{eq2.3} follows from  
\cite[p. 365, Thm. 16.2.9, formula (16.2.7)]{EG17}), while \eqref{eq2.5}
is a consequence of \cite[Prop. 2]{DDS05}. For convenience of the reader, we repeat the proofs.

\begin{proof}
Let $\NN_{\alpha}$ denote the $A$-module on the right-hand side of \eqref{eq2.3}. We prove the inclusions
$\NN_{\alpha}\subseteq A_{\alpha}\subseteq A[\alpha ]\cap A[\alpha^{-1}]\subseteq\NN_{\alpha}$.

First observe that if $1\leq i\leq n-1$, $0\leq j\leq n-1$, then
\begin{align*}
\omega_i\alpha^j &=
\sum_{k=0}^{i-1} a_k\alpha^{i+j-k}\in\MM_{\alpha}\ \ \text{if } i+j\leq n-1,
\\
\omega_i\alpha^j &= (\omega_i -f_{\alpha}(\alpha))\alpha^j=
-\sum_{k=i}^{n} a_k\alpha^{i+j-k}\in\MM_{\alpha}\ \ \text{if } i+j\geq n,
\end{align*}
implying $\NN_{\alpha}\subseteq A_{\alpha}$.

Second,  $A_{\alpha}\subseteq \MM_{\alpha}\cap\alpha^{1-n}\MM_{\alpha}\subseteq A[\alpha ]\cap A[\alpha^{-1}]$.

Third, let $\xi =P(\alpha )=Q(\alpha^{-1})\in A[\alpha ]\cap A[\alpha^{-1}]$, where $P,Q\in A[X]$.
We prove by induction on $\deg P$, that $\xi\in\NN_{\alpha}$. For $\deg P=0$ this is clear.
Let $\deg P=r\geq 1$. Consider the polynomial $H(X):=X^{\deg Q}P(X)-X^{\deg Q}Q(X^{-1})\in A[X]$. 
The polynomial $H$ is non-zero, since otherwise $P(X)=Q(X^{-1})$, which is impossible.
Let $b$ be the leading coefficient of $P$. Then $b$ is also the leading coefficient of $H$.
Since $H(\alpha )=0$, $f_{\alpha}$ must divide $H$ in $\vk [X]$.
But by assumption, the coefficients of $f_{\alpha}$ have gcd $1$, so by Gauss' Lemma 
$f_{\alpha}$ divides $H$ in $A[X]$, in particular, the leading coefficient $a_0$ of $f_{\alpha}$
divides $b$. Now if $r\geq n$, we have $P(\alpha )=P^*(\alpha )$
where $P^*(X)=P(X)-(b/a_0)X^{r-n}f_{\alpha}(X)$ is a polynomial in $A[X]$
of degree $<r$ and we can apply the induction hypothesis. 
If $r<n$, then $P(\alpha )=(b/a_0)\omega_r+P^*(\alpha )$, where $P^*\in A[X]$ has degree $<r$. We know already that $\omega_r\in A[\alpha ]\cap A[\alpha^{-1}]$,
so $P^*(\alpha )\in A[\alpha ]\cap A[\alpha^{-1}]$. We can again apply the induction hypothesis.
\end{proof}

Let $\MM$ be an $A$-submodule of $A_K$ with basis $\gamma_1\kdots \gamma_n$, say,
where $n=[K:\vk ]$.
The discriminant ideal $\fd_{\MM /A}$ of $\MM$ over $A$ is defined as the ideal of $A$ generated
by $D_{K/\vk}(\gamma_1\kdots\gamma_n):=\Big(\det (\gamma_i^{(j)})_{i,j=1\kdots n}\Big)^2$.
This does not depend on the choice of basis.

\begin{lemma}\label{lem2.-1}
Let $\alpha\in K$ with $\vk (\alpha )=K$ and let $f_{\alpha}=a_0X^n+\cdots +a_n\in A[X]$
be a primitive minimal polynomial of $\alpha$. Then $\fd_{A_{\alpha }/A}=D(f_{\alpha})A$,
where 
 $D(f_{\alpha})=a_0^{2n-2}\prod_{1\leq i<j\leq n}(\alpha^{(i)}-\alpha^{(j)})^2$.
\end{lemma}

\begin{proof} Same reasoning as \eqref{eq1.4}.
\end{proof} 

For $\alpha_1\kdots \alpha_r\in L$, denote by $[\alpha_1\kdots \alpha_r]$ the fractional ideal
of $A_L$, i.e., $A_L$-module, generated by $\alpha_1\kdots \alpha_r$. Further,
for a finitely generated $A$-submodule $\MM$ of $K$ and for distinct $i,j\in\{ 1\kdots n\}$,
let $\fd_{ij}(\MM )$ be the fractional ideal of $A_L$ generated by 
$\xi^{(i)}-\xi^{(j)}$
for all $\xi\in\MM$. Thus, if $\MM$ is generated as an $A$-module by $\xi_1\kdots\xi_r$,
we have
\begin{equation}\label{eq2.6}
\fd_{ij}(\MM )=[\xi_1^{(i)}-\xi_1^{(j)}\kdots \xi_r^{(i)}-\xi_r^{(j)}].
\end{equation}

\begin{lemma}\label{lem2.1}
Let $\alpha$ be such that $K=\vk (\alpha )$ and $i,j\in\{ 1\kdots n\}$ with $i\not= j$.
Then
\[
[\alpha^{(i)}-\alpha^{(j)}]= [1,\alpha^{(i)}]\cdot [1,\alpha^{(j)}]\cdot 
\fd_{ij}(A_{\alpha}).
\]
\end{lemma}

\begin{proof} (cf. \cite[Lemma 17.6.4]{EG17})
Let $\omega_1\kdots\omega_{n-1}$ be as in \eqref{eq2.3}. Then
\[
\alpha f_{\alpha}(X)=(X-\alpha )(\omega_1 X^{n-1}+\omega_2X^{n-2}+\cdots +\omega_n),
\]
where $\omega_n:=-a_n$.
This implies
\begin{align*}
&(\alpha^{(i)}-\alpha^{(j)})Xf_{\alpha}(X)
\\
&\quad =(X-\alpha^{(j)})\alpha^{(i)}f_{\alpha}(X)-(X-\alpha^{(i)})\alpha^{(j)}f_{\alpha}(X)
\\
&\quad =(X-\alpha^{(i)})(X-\alpha^{(j)})\cdot \big( (\omega_1^{(i)}-\omega_1^{(j)})X^{n-1}+\cdots +
(\omega_{n-1}^{(i)}-\omega_{n-1}^{(j)})\big).
\end{align*}
We apply Gauss' lemma for Dedekind domains, which in our case asserts that if $g_1,g_2\in L[X]$
then $[g_1g_2]=[g_1]\cdot [g_2]$, where $[g]$ is the fractional ideal of $A_L$
generated by the coefficients of $g\in L[X]$.
Using that the coefficients of $f_{\alpha}$ have gcd $1$, together with
\eqref{eq2.3}, \eqref{eq2.6}, we obtain
\begin{align*}
[\alpha^{(i)}-\alpha^{(j)}] &=[1,\alpha^{(i)}]\cdot [1,\alpha^{(j)}]
\cdot [\omega_1^{(i)}-\omega_1^{(j)}\kdots \omega_{n-1}^{(i)}-\omega_{n-1}^{(j)}]
\\
&=[1,\alpha^{(i)}]\cdot [1,\alpha^{(j)}]\cdot \fd_{ij}(A_{\alpha}).
\end{align*}
\end{proof}

If $[K:\vk ]=n\geq 4$ then for
$\alpha$ with $K=\vk (\alpha )$ and pairwise distinct $i,j,k,l\in\{ 1\kdots n\}$, we define the \emph{cross ratio}
\begin{equation}\label{eq4.crossratio}
\cross_{ijkl}(\alpha ):=
\frac{(\alpha^{(i)}-\alpha^{(j)})(\alpha^{(k)}-\alpha^{(l)})}
{(\alpha^{(i)}-\alpha^{(k)})(\alpha^{(j)}-\alpha^{(l)})}.
\end{equation}

\begin{lemma}\label{lem2.2}
Suppose $[K:\vk ]=n\geq 4$.
Let $\alpha ,\beta$ be such that $\vk (\alpha )=\vk (\beta )=K$ and $A_{\alpha}=A_{\beta}$.
Then for all pairwise distinct $i,j,k,l\in\{ 1\kdots n\}$ we have
\[
\frac{\cross_{ijkl}(\alpha )}{\cross_{ijkl}(\beta )}\in A_L^*.
\]
\end{lemma}

\begin{proof}
Lemma \ref{lem2.1} implies $[\cross_{ijkl}(\alpha )]=[\cross_{ijkl}(\beta )]$ for all $i,j,k,l$.
\end{proof}

\begin{lemma}\label{lem2.3}
Let $K$ be a finite extension of $\vk$, and
let $\alpha,\beta$ be such that $\vk (\alpha )=\vk (\beta )=K$.
\\[0.1cm]
(i) Suppose that $[K:\vk ]=3$. Then $\alpha ,\beta$ are $\GL_2(\vk )$-equivalent.
\\[0.1cm]
(ii) Suppose $[K:\vk ]=n\geq 4$.
Then $\alpha ,\beta$ are $\GL_2(\vk )$-equivalent 
if and only if
$\cross_{ijkl}(\alpha )=\cross_{ijkl}(\beta )$ for all pairwise distinct $i,j,k,l\in\{ 1\kdots n\}$.
\end{lemma}

\begin{proof} (cf. 
\cite[Lemma 17.7.2]{EG17})
(ii) From elementary projective geometry, we know that
$\cross_{ijkl}(\alpha )=\cross_{ijkl}(\beta )$ for all pairwise distinct $i,j,k,l\in\{ 1\kdots n\}$
if and only if there is $C\in\GL_2(L)$ such that $\beta^{(i)}=C\alpha^{(i)}$ for $i=1\kdots n$.
Suppose the latter to be the case. Then since $n\geq 4$, the matrix $C$ is determined uniquely up to a scalar.
Clearly, we have $\beta^{(i)}=\sigma (C)\alpha^{(i)}$ for $i=1\kdots n$ and every $\sigma\in{\rm Gal}(L/\vk )$.
If we assume that one of the entries of $C$ is $1$, then $\sigma (C)=C$ for every $\sigma\in{\rm Gal}(L/\vk )$,
i.e., $C\in\GL_2(\vk )$.

(i) By elementary projective geometry, there is an up to a scalar factor unique
$C\in\GL_2(L)$ such that $\beta^{(i)}=C\alpha^{(i)}$ for $i=1,2,3$. If we take
$C$ such that one of its entries is $1$ then similarly as above
it follows that $C\in\GL_2(\vk )$.
\end{proof}

\begin{lemma}\label{lem2.4}
Assume that $[K:\vk ]\geq 3$.
Let $\alpha ,\beta$ be such that $\vk (\alpha )=\vk (\beta )=K$ and $A_{\alpha}=A_{\beta}$.
Suppose that $\alpha$, $\beta$ are $\GL_2(\vk )$-equivalent.
Then $\alpha ,\beta$ are $\GL_2(A)$-equivalent.
\end{lemma}

\begin{proof} (cf. \cite[Proposition 17.6.5]{EG17}) 
Since $A$ is a principal ideal domain, we may assume that $\beta =C\alpha$, where the entries of $C$ belong to $A$ and have gcd $1$. Further, $C$ can be put into Smith Normal Form,
i.e., there are matrices $U,V\in\GL_2(A)$ such that 
$UCV=\big(\begin{smallmatrix}a&0\\ 0&1\end{smallmatrix}\big)$ with $a\in A\setminus\{ 0\}$.
Let $\beta_1:=U\beta$, $\alpha_1:=V^{-1}\alpha$. Then since $\alpha ,\beta\not\in\vk$
we have $\alpha_1,\beta_1\not=\infty$ and moreover, $A_{\alpha_1}=A_{\beta_1}$
and $\beta_1=a\alpha_1$. We have to show that $a\in A^*$.

Let $f_{\alpha_1}(X)=a_0X^n+\cdots +a_n\in A[X]$ be a primitive minimal polynomial of $\alpha_1$, i.e., with 
$\gcd (a_0\kdots a_n)=1$. Then $\beta_1$ has primitive minimal polynomial
\[ 
f_{\beta_1}(X)=\lambda f_{\alpha_1}(X/a)=\lambda (a^{-n}a_0X^n+a^{1-n}a_1X^{n-1}+\cdots +a_n),
\]
where $\lambda\in\vk$ is such that the coefficients of $f_{\beta_1}$ are in $A$
and have gcd $1$. By \eqref{eq2.3}, $A_{\alpha_1}$ is a free $A$-module with basis 
$1$, $\omega_1\kdots \omega_{n-1}$ with $\omega_i=\sum_{k=0}^{i-1} a_k\alpha_1^{i-k}$
for $i=1\kdots n-1$.
By replacing $\alpha_1$ with $\beta_1=a\alpha_1$, and $a_i$ by $\lambda a^{i-n}a_i$, we see that $A_{\beta_1}$ has basis $1,\lambda a^{1-n}\omega_1,\lambda a^{2-n}\omega_2\kdots \lambda a^{-1}\omega_{n-1}$. Since $A_{\alpha_1}=A_{\beta_1}$, this must imply
\[
\lambda a^{i-n}\in A^*\ \ \text{for } i=1\kdots n-1,
\]
hence $a\in A^*,\lambda\in A^*$.
\end{proof}

\section{Application of unit equations}\label{section3}

Let $K$ be a number field of degree $n\geq 4$ and $L$ its normal closure.
In the case $n=4$ we don't impose any constraints on $L$,
while for $n\geq 5$ we assume that ${\rm Gal}(L/\Qq )$ is $5$-transitive.

We call $\alpha_1\in K$ $k$-\emph{special} if $K=\Qq (\alpha_1 )$ and there are $\alpha_2\kdots \alpha_k$ such that $\Zz_{\alpha_1}=\cdots =\Zz_{\alpha_k}$ and $\alpha_1\kdots \alpha_k$ are pairwise $\GL_2(\Zz )$-inequivalent.
We call $\alpha_1$ \emph{special} if it is $2$-special.

Theorem \ref{thm1.1} follows,
once we have shown that in the case $n=4$, the $3$-special numbers of $K$ lie in only finitely
many $\GL_2 (\Zz )$-equivalence classes, and in the case $n\geq 5$ that the special numbers of $K$
lie in only finitely many $\GL_2(\Zz )$-equivalence classes. Indeed,
the orders of $K$ with $k$ rational monogenizations are all of the shape $\Zz_{\alpha}$ where
$\alpha$ is $k$-special, and if such $\alpha$ lie in only finitely many $\GL_2(\Zz )$-equivalence classes,
there are only finitely many orders $\Zz_{\alpha}$.

In the present section we prove the following proposition.
Here, we apply some results from the theory of unit equations.

\begin{proposition}\label{prop3.1}
(i) Let $K$ be a quartic number field. Then the set of $3$-special numbers of $K$
is contained in finitely many $\GL_2(\Qq )$-equivalence classes.
\\[0.1cm]
(ii) Let $K$ be a number field of degree $n\geq 5$ such that the Galois group of its normal closure $L$ is $5$-transitive. 
Then the set of special numbers of $K$
is contained in finitely many $\GL_2(\Qq )$-equivalence classes.
\end{proposition}

We will show later (see Proposition \ref{prop5.1} below) that
a $\GL_2(\Qq )$-equivalence class of special numbers is the union of finitely many 
$\GL_2(\Zz )$-equivalence classes.

We start with some initial observations. Let $\alpha ,\beta\in K$ with $\Qq (\alpha )=\Qq (\beta )=K$,
$\Zz_{\alpha}=\Zz_{\beta }$ and $\alpha ,\beta$ $\GL_2(\Zz )$-inequivalent.
Then
\begin{equation}\label{eq3.1}
\begin{array}{c}
\cross_{ijkl}(\alpha )\not=\cross_{ijkl}(\beta )
\\[0.1cm]
\text{for all pairwise distinct $i,j,k,l\in\{ 1\kdots n\}$.}
\end{array}
\end{equation}
Indeed,
suppose that for some tuple $(i,j,k,l)$ we have equality, say $(1,2,3,4)$. 
In the case $n=4$ this implies equality for each permutation $(i,j,k,l)$ of $(1,2,3,4)$
since $\cross_{ijkl}(\cdot )$ is a fractional linear transformation
of $\cross_{1234}(\cdot )$.
In the case $n\geq 5$, we obtain equality for all $i,j,k,l$
since by our assumption on the normal closure $L$, there is $\sigma\in{\rm Gal}(L/\Qq )$
that maps $\cross_{1234}(\cdot )$ to $\cross_{ijkl}(\cdot )$.
Lemma \ref{lem2.3} now implies that $\alpha ,\beta$ are $\GL_2(\Qq )$-equivalent, and subsequently Lemma \ref{lem2.4} that $\alpha ,\beta$ are $\GL_2(\Zz )$-equivalent,
contrary to our assumption.

Another important observation is the identity for cross ratios
\begin{equation}\label{eq3.2}
\cross_{ijkl}(\alpha )+\cross_{ilkj}(\alpha )=1
\end{equation}
for all $\alpha\in K$ and all pairwise distinct $i,j,k,l\in\{ 1\kdots n\}$.
Now let $\alpha, \beta$ be such that $\Qq (\alpha )=\Qq (\beta )=K$
and $\Zz_{\alpha}=\Zz_{\beta}$. Put
\[
\varepsilon_{ijkl}:=\frac{\cross_{ijkl}(\beta )}{\cross_{ijkl}(\alpha )};
\]
then from \eqref{eq3.2} and Lemma \ref{lem2.2}
we deduce
\begin{equation}\label{eq3.3}
\cross_{ijkl}(\alpha)\cdot \varepsilon_{ijkl}+\cross_{ilkj}(\alpha )\cdot \varepsilon_{ilkj}=1,\ \ 
\varepsilon_{ijkl}\in\OO_L^*,\ \varepsilon_{ilkj}\in\OO_L^*,
\end{equation}
where $\OO_L$ is the ring of integers of $L$. This allows us to apply the theory of unit equations.

We first prove part (i), and then part (ii).

\begin{proof}[Proof of part (i) of Proposition \ref{prop3.1}]
Let $K$ be a quartic number field, and let $\alpha\in K$ be $3$-special. Choose $\beta ,\gamma\in K$ such that
$\alpha ,\beta ,\gamma$ are pairwise $\GL_2(\Zz )$-inequivalent,
and $\Zz_{\alpha}=\Zz_{\beta}=\Zz_{\gamma}$. Put
\[
\varepsilon_{ijkl}:=\frac{\cross_{ijkl}(\beta )}{\cross_{ijkl}(\alpha )},\ \ 
\eta_{ijkl}:=\frac{\cross_{ijkl}(\gamma )}{\cross_{ijkl}(\alpha )}
\]
for each permutation $(i,j,k,l)$ of $(1,2,3,4)$.

By \eqref{eq3.1}--\eqref{eq3.3}, the pairs $(1,1)$, $(\varepsilon_{1234},\varepsilon_{1432})$,
$(\eta_{1234},\eta_{1432})$ are three distinct solutions to the equation
\begin{equation}\label{eq3.0}
\cross_{1234}(\alpha )x+\cross_{1432}(\alpha )y=1\ \ \text{in } x,y\in\OO_L^*.
\end{equation}
We now apply the following result on unit equations.
\footnote{Equations with unknowns from a multiplicative group $\Gamma$ of finite rank are often called
`unit equations' since in most applications, $\Gamma$ is the unit group of a domain.}

\begin{lemma}\label{lem3.2}
Let $F$ be a field of characteristic $0$ and $\Gamma$ a subgroup of $F^*$ of finite rank.
Then there are only finitely many pairs $(a,b)\in F^*\times F^*$ with $a+b=1$
such that the equation
\[
ax+by=1\ \ \text{in } x,y\in\Gamma
\]
has more than two solutions, the pair $(1,1)$ included.
\end{lemma}

\begin{proof}
This is essentially a result of Gy\H{o}ry, Stewart, Tijdeman, and the author \cite[Thm. 1]{EGST88}, see also \cite[Thm. 6.1.6]{EG15}. Their proof uses a finiteness result
for linear unit equations in several unknowns, which in turn follows from Schmidt's Subspace Theorem.
\end{proof}

We continue with the proof of part (i) of Proposition \ref{prop3.1}. Since \eqref{eq3.0} has three distinct solutions in $\OO_L^*$ including $(1,1)$ and $\OO_L^*$ is finitely generated, Lemma \ref{lem3.2} implies that if $\alpha$ runs through the $3$-special numbers of $K$,
then $\cross_{1234}(\alpha )$ runs through a finite set. 
If $(i,j,k,l)$ is a permutation of $(1,2,3,4)$, then $\cross_{ijkl}(\cdot )$
is a fractional linear transformation of $\cross_{1234}(\cdot )$,
hence $\cross_{ijkl}(\alpha )$ runs through a finite set as well. 
Now Lemma \ref{lem2.3}(ii) implies that the $3$-special numbers
$\alpha\in K$ lie in only finitely many $\GL_2(\Qq )$-equivalence classes.
\end{proof} 

\begin{proof}[Proof of part (ii) of Proposition \ref{prop3.1}]
Let $K$ be a number field of degree $n\geq 5$ 
such that the Galois group of its normal closure $L$
is $5$-transitive.
Take a special $\alpha\in K$. 
Choose $\beta$ such that $\Zz_{\beta}=\Zz_{\alpha}$. Recall that by
Lemma \ref{lem2.2}
\[
\ve_{ijkl}:=\frac{\cross_{ijkl}(\beta )}{\cross_{ijkl}(\alpha )}\in \OO_L^*
\]
for all pairwise distinct $i,j,k,l\in\{1\kdots n\}$.
Viewing \eqref{eq3.2} and \eqref{eq3.3} as linear equations in $\cross_{ijkl}(\alpha )$
and $\cross_{lijk}(\alpha )$ we derive from Cramer's rule,
\begin{equation}\label{eq3.4}
\cross_{ijkl}(\alpha )=\frac{\ve_{ilkj}-1}{\ve_{ilkj}-\ve_{ijkl}}.
\end{equation}
Our strategy is as follows. Using algebraic relations between the $\ve_{ijkl}$
and finiteness results for unit equations, 
we show that if $\alpha$ runs through the
special numbers of $K$, then one of the $\ve_{ijkl}$, say $\ve_{1234}$, runs through a finite set. Our assumption
that ${\rm Gal}(L/\Qq )$ is $5$-transitive implies that the numbers 
$\ve_{ijkl}$ are all
conjugate to one another, thus it follows that $\ve_{ijkl}$ runs through a finite set
for all $i,j,k,l$. But then, \eqref{eq3.4} implies that $\cross_{ijkl}(\alpha )$ runs through
a finite set for all $i,j,k,l$. Finally,
Lemma \ref{lem2.3}(ii) implies that the special numbers
$\alpha\in K$ lie in only finitely many $\GL_2(\Qq )$-equivalence classes.

We first collect some algebraic relations between the $\ve_{ijkl}$. 
It is straightforward to verify
\begin{equation}\label{eq3.5}
\left\{
\begin{array}{l}
\ve_{ijkl}=\ve_{jilk}=\ve_{klij}=\ve_{lkji},
\\[0.1cm]
\ve_{ijkl}^{-1}=\ve_{ikjl},
\\[0.1cm]
\frac{\ve_{ijkl}}{\ve_{ijlk}}=\ve_{ilkj}
\end{array}\right.
\end{equation}
for all pairwise distinct $i,j,k,l\in\{ 1\kdots n\}$ and moreover,
\begin{equation}\label{eq3.6}
\frac{\ve_{ijkl}}{\ve_{ijkm}}=\ve_{jmlk}
\end{equation}
for all pairwise distinct $i,j,k,l,m\in\{ 1\kdots n\}$. 

We derive a few more relations.
From \eqref{eq3.4} and \eqref{eq3.5} it follows that $\cross_{ijkl}(\beta )=\ve_{ijkl}\cross_{ijkl}(\alpha )=\frac{\ve_{ilkj}-1}{\ve_{iljk}-1}$.
Picking a fifth index $m$, we get
\[
1 =\frac{\cross_{jmlk}(\beta )\cross_{ijkm}(\beta )}{\cross_{ijkl}(\beta)}
\\
=\frac{\ve_{jklm}-1}{\ve_{jkml}-1}\cdot
\frac{\ve_{imkj}-1}{\ve_{imjk}-1}\cdot\frac{\ve_{iljk}-1}{\ve_{ilkj}-1}.
\]
We apply this with $(i,j,k,l,m)=(5,1,2,3,4)$. Thus, we obtain
\begin{equation}\label{eq3.8}
(\ve_{1234}-1)(\ve_{1245}-1)(\ve_{1253}-1)=(\ve_{1243}-1)(\ve_{1254}-1)(\ve_{1235}-1),
\end{equation}
where, as mentioned before, all entries belong to $\OO_L^*$.
We apply the following result.

\begin{lemma}\label{lem3.3}
Let $F$ be a field of characteristic $0$ and $\Gamma$ a subgroup of $F^*$
of finite rank. Consider the equation
\begin{align}\label{eq3.9}
(x_1-1)(x_2-1)(x_3-1)=\,&(y_1-1)(y_2-1)(y_3-1)
\\
\notag
&\text{  in } x_1,x_2,x_3,y_1,y_2,y_3\in\Gamma.
\end{align}
There is a finite subset $\SS$ of $\Gamma$ such that every solution of \eqref{eq3.9}
satisfies one of the following:
\begin{itemize}
\item[(a)] at least one of $x_1\kdots y_3$ belongs to $\SS$;
\item[(b)]
there are $s_1,s_2,s_3\in\{ \pm 1\}$ such that $(x_1,x_3,x_3)$ is a permutation of $(y_1^{s_1},y_2^{s_2},y_3^{s_3})$;
\item[(c)] at least one of the numbers in $\{ x_ix_j, x_i/x_j ,y_iy_j, y_i/y_j: 1\leq i<j\leq 3\}$
is either $-1$ or a primitive cube root of unity.
\end{itemize}
\end{lemma}

\begin{proof} 
This is a result of B\'{e}rczes, Gy\H{o}ry, and the author \cite[Prop. 8.1]{BEG13}.
They deduced 
the above lemma from a finiteness result
for linear unit equations in several unknowns, and so again Schmidt's Subspace Theorem
is at the background.
\end{proof}

We apply Lemma \ref{lem3.3} with $\Gamma =\OO_L^*$ to \eqref{eq3.8}. We show that each of the three cases (a), (b), (c) gives rise to only finitely many possible values for $\ve_{1234}$.
Recall that we assume that ${\rm Gal}(L/\Qq )$
is $5$-transitive.
Hence for any two quintuples of distinct indices $(i,j,k,l,m)$ and $(i',j',k',l',m')$,
there is $\sigma\in {\rm Gal}(L/\Qq )$ mapping $\alpha^{(i)},\beta^{(i)}$,$\ldots$,$\alpha^{(m)},\beta^{(m)}$
to $\alpha^{(i')},\beta^{(i')}$,$\ldots$,$\alpha^{(m')},\beta^{(m')}$, respectively.
Consequently, any two $\ve_{ijkl}$, $\ve_{i',j',k',l'}$ are conjugate to each other.
Similarly, 
from an identity between $\ve$-s with indices from a quintuple $(i,j,k,l,m)$
we can derive a similar identity with indices from $(i',j',k',l',m')$  
by applying a suitable element of ${\rm Gal}(L/\Qq )$. 

The above observations imply that
if we have shown that 
one of the $\ve_{ijkl}$ runs through  a finite set, then so does
$\ve_{1234}$. This settles case (a). As for (b) and (c),
using again the above observations, we are left with the following subcases.
Let $\TT$ denote the group of $6$-th roots of unity in $L$.
\\[0.2cm]
{\bf Case b1.} $\ve_{1234}=\ve_{1243}$.
\\
Then by \eqref{eq3.5}, $\ve_{1432}=\frac{\ve_{1234}}{\ve_{1243}}=1$, 
which by conjugacy implies $\ve_{1234}=1$.
\\[0.2cm]
{\bf Case b2.} $\ve_{1234}=\ve_{1243}^{-1}$.
\\
By \eqref{eq3.5}, $\ve_{1243}^{-1}=\ve_{1234}^{-1}\ve_{1432}$, so $\ve_{1234}^2=\ve_{1432}$.
By conjugacy, we may interchange the indices $2,3$, while keeping $1$ and $4$ fixed, 
so we have also $\ve_{1324}=\ve_{1342}^{-1}$. Applying again \eqref{eq3.5},
this gives $\ve_{1234}=\ve_{1432}^{-1}$. Hence $\ve_{1234}^3=1$.
\\[0.2cm]
{\bf Case b3.} $\ve_{1234}=\ve_{1254}$.
\\
By \eqref{eq3.5}, \eqref{eq3.6}, $1=\frac{\ve_{1234}}{\ve_{1254}}=\frac{\ve_{2143}}{\ve_{2145}}=\ve_{1534}$. By conjugacy, $\ve_{1234}=1$.
\\[0.2cm]
{\bf Case b4.} $\ve_{1234}=\ve_{1254}^{-1}$.
\\
By conjugacy, we may interchange $2$ and $3$, keeping $1,4,5$ fixed, so we have $\ve_{1324}=\ve_{1354}^{-1}$,
which together with \eqref{eq3.5} implies $\ve_{1234}=\ve_{1354}$.
From \eqref{eq3.5} and \eqref{eq3.6} we deduce $\frac{\ve_{1254}}{\ve_{1354}}=\ve_{1234}$.
Multiplying these relations together, we obtain $\ve_{1234}^3=1$.
\\[0.2cm]
{\bf Case c1.} $\ve_{1234}\cdot \ve_{1245}\in\TT$.
\\
By interchanging $3$ and $5$, keeping $1,2,4$ fixed, we see that $\ve_{1254}\cdot \ve_{1243}\in\TT$.
Using \eqref{eq3.5}, \eqref{eq3.6}, we get
\[
\frac{\ve_{1234}\cdot \ve_{1245}}{\ve_{1254}\cdot\ve_{1243}}=\frac{\ve_{1534}}{\ve_{2534}}=\ve_{1532}\in \TT
\]
and by conjugacy, $\ve_{1234}\in\TT$.
\\[0.2cm]
{\bf Case c2.} $\frac{\ve_{1234}}{\ve_{1245}}\in\TT$.
\\
Interchanging $2$ and $3$, keeping $1,4,5$ fixed, we obtain 
$\frac{\ve_{1324}}{\ve_{1345}}\in\TT$, and then, using $\ve_{1324}=\ve_{1234}^{-1}$,
$\ve_{1245}\cdot\ve_{1345}\in\TT$. 
By taking conjugates, we get $\ve_{1234}\cdot\ve_{1235}\in\TT$,
and also $\ve_{1234}\cdot\ve_{5234}\in\TT$. Applying \eqref{eq3.6},
the latter yields $\ve_{1234}\cdot \frac{\ve_{1234}}{\ve_{1235}}\in\TT$.
Hence $\ve_{1234}^3\in\TT$.
\\[0.2cm]

As mentioned above, this completes the proof of Proposition \ref{prop3.1}.
\end{proof}

\section{Investigation of $\GL_2(\vk )$-classes}\label{section4}

Let $K$ be a number field of degree $\geq 4$. In the next section we show (Proposition \ref{prop5.1})
that each $\GL_2(\Qq )$-equivalence class of special numbers in $K$ is the union 
of finitely many $\GL_2(\Zz )$-equivalence classes. Together with Proposition \ref{prop3.1}
this will imply Theorem \ref{thm1.1}. In the present section, we develop some
machinery needed for the proof of Proposition \ref{prop5.1}.
We have worked out this machinery for arbitrary principal ideal domains 
of characteristic $0$ so that we can use it also in Section \ref{section6}
where we will prove a generalization of Theorem \ref{thm1.1} over
rings of $S$-integers of number fields. 

Let $A$ be a principal ideal domain of characteristic $0$, $\vk$ its field of fractions,
$K$ an extension of $\vk$ of degree $n\geq 4$, and $L$ the normal closure of $K$ over $\vk$.
We consider so-called \emph{special pairs} in $K$, i.e., pairs
$(\alpha ,\beta )$ such that $\vk (\alpha )=\vk (\beta )=K$, $A_{\alpha}=A_{\beta}$ 
and $\alpha ,\beta$
are $\GL_2(A)$-inequivalent. 
Two special pairs $(\alpha ,\beta )$ and $(\alpha^*,\beta^*)$
are called $\GL_2(\vk )$-equivalent if $\alpha^*$ is $\GL_2(\vk )$-equivalent to $\alpha$
and $\beta^*$ is $\GL_2(\vk )$-equivalent to $\beta$.

Let $(\alpha ,\beta )$, $(\alpha^*,\beta^*)$ be two $\GL_2(\vk )$-equivalent
special pairs. Then since we are working over a principal ideal domain $A$,
\begin{equation}\label{eq4.1}
\alpha^*=C\alpha,\ \ \beta^*=C'\beta ,
\end{equation}
where $C=\begin{pmatrix}a&b\\c&d\end{pmatrix}$, $C'=\begin{pmatrix}a'&b'\\c'&d'\end{pmatrix}$,
with
\begin{align*}
&a,b,c,d\in A,\ \ \gcd(a,b,c,d)=1,\ \ \Delta :=ad-bc\not= 0,
\\
&a',b',c',d'\in A,\ \ \gcd (a',b',c',d')=1,\ \ \ \Delta':=a'd'-b'c'\not= 0.
\end{align*}
Recall that by Lemma \ref{lem2.2}
we have $\frac{\cross_{ijkl}(\beta )}{\cross_{ijkl}(\alpha )}\in A_L^*$
for all pairwise distinct $i,j,k,l\in\{ 1\kdots n\}$.

\begin{proposition}\label{prop4.1}
Let $\fd$ be the discriminant ideal of $A_{\alpha}$, 
and let $\fa (\alpha ,\beta )$ denote the ideal of $A_L$ generated by all numbers
$\frac{\cross_{ijkl}(\beta )}{\cross_{ijkl}(\alpha )}-1$ for all pairwise distinct $i,j,k,l\in\{ 1\kdots n\}$. Then  
\begin{equation}\label{eq4.0}
\Delta A_L\supseteq
\fd^{5}\cdot\fa (\alpha ,\beta )^{2}.
\end{equation}
\end{proposition}
\noindent
Recall that by Lemmas \ref{lem2.3} and \ref{lem2.4}, the ideal $\fa (\alpha ,\beta )$
is not zero. We mention that our proof implies also that
$\Delta/\Delta'\in A^*$, but this will not be needed.

We start with some preparations and then prove two lemmas, which together imply
Proposition \ref{prop4.1}.

Let $C,C'$ be the matrices from \eqref{eq4.1}.
Since $A$ is a principal ideal domain,
there are matrices $U,V,U',V'\in\GL_2(A)$ such that 
\[
UCV=\begin{pmatrix}\Delta&0\\ 0&1\end{pmatrix},\ \ U'C'V'=\begin{pmatrix}\Delta'&0\\ 0&1\end{pmatrix}.
\]
Put $\alpha_1:= V^{-1}\alpha $, $\alpha_1^*:=U\alpha^*$, $\beta_1:=V'^{-1}\beta$, 
$\beta_1^*:=U'\beta$. Then $\alpha_1^*=\Delta\alpha_1$, $\beta_1^*=\Delta'\beta_1$,
$A_{\alpha_1}=A_{\beta_1}$, $A_{\alpha_1^*}=A_{\beta_1^*}$,
$(\alpha_1 ,\beta_1)$, $(\alpha_1^*,\beta_1^*)$ are $\GL_2(\vk )$-equivalent
special pairs, and $\fa (\alpha_1,\beta_1)=\fa (\alpha ,\beta )$. 
So in the proof of Proposition \ref{prop4.1} we may
replace $\alpha$, $\alpha^*$, $\beta$, $\beta^*$ by $\alpha_1$, $\alpha_1^*$,
$\beta_1$, $\beta_1^*$, in other words, without loss of generality we may assume
\begin{equation}\label{eq4.2}
\alpha^*=\Delta\alpha,\ \ \beta^*=\Delta'\beta .
\end{equation}

So assume \eqref{eq4.2}. Let
\[
f_{\alpha}=a_0X^n+\cdots +a_n,\ \ f_{\beta}=b_0X^n+\cdots +b_n
\]
be primitive minimal polynomials of $\alpha ,\beta$. 
By Lemma \ref{lem2.0}, the ring $A_{\alpha} = A_{\beta}$ has $A$-module bases
\[
\{1,\omega_1\kdots \omega_{n-1}\},\ \ \{ 1,\rho_1\kdots \rho_{n-1}\}
\]
respectively, where 
\begin{equation}\label{eq4.4}
\omega_i=\sum_{j=0}^{i-1} a_j\alpha^{i-j},\ \ \rho_i=\sum_{j=0}^{i-1} b_j\beta^{i-j}
\ \ (i=1\kdots n-1).
\end{equation}
Hence there are a matrix $M=(m_{ij})_{i,j=1\kdots n-1}\in\GL_{n-1}(A)$
and $m_{i,0}\in A$ ($i=1\kdots n-1$) such that
\begin{equation}\label{eq4.5}
\rho_i=m_{i,0}+\sum_{j=1}^{n-1} m_{ij}\omega_j\ \ \text{for } i=1\kdots n-1.
\end{equation}
Let us write $[\ldots ]$ for the fractional ideal of $A$ generated by the
elements between the brackets.

\begin{lemma}\label{lem4.2}
The following holds:
\begin{align}
\label{eq4.-1}
&[\Delta ]=[\Delta'],
\\
\label{eq4.9}
&\mod{m_{ij}}{0}{\Delta^{j-i}}\ \ \text{for } i=1\kdots n-1,\, j>i,
\\
\label{eq4.10}
&\gcd (m_{ii},\Delta )=1\ \ \text{for } i=1\kdots n-1,
\\
\label{eq4.11}
&[a_0,\Delta ]=[b_0,\Delta ].
\end{align}
\end{lemma}

\begin{proof}
By \eqref{eq4.2}, there are non-zero $\lambda ,\mu\in\vk$, 
such that $\alpha^*$, $\beta^*$ have primitive minimal polynomials
\begin{equation}\label{eq4.3}
\begin{split}
f_{\alpha^*} &=\lambda (a_0\Delta^{-n}X^n +a_1\Delta^{1-n}X^{n-1}+\cdots +a_n),
\\ 
f_{\beta^*} &=\mu (b_0\Delta'^{-n}X^n+b_1\Delta'^{1-n}X^{n-1}+\cdots +b_n).
\end{split}
\end{equation}
From \eqref{eq4.2}, \eqref{eq4.3} it follows that $A_{\alpha^*}=A_{\beta^*}$ has $A$-module bases
\begin{equation}\label{eq4.6}
\{ 1, \lambda\Delta^{1-n}\omega_1\kdots \lambda\Delta^{-1}\omega_{n-1}\},\ \ 
\{ 1, \mu\Delta'^{1-n}\rho_1\kdots \mu\Delta'^{-1}\rho_{n-1}\}.
\end{equation}
Hence there are $M^*=(m_{ij}^*)_{i,j=1\kdots n-1}\in\GL_{n-1}(A)$ and $m_{i,0}^*\in A$
($i=1\kdots n-1$) such that
\[
\mu\Delta'^{i-n}\rho_i=m_{i,0}^*+\sum_{j=1}^{n-1} m_{ij}^*\lambda\Delta^{j-n}\omega_j\ \ \text{for } i=1\kdots n-1.
\]
A comparison with \eqref{eq4.5} gives
\[
\mu\Delta'^{i-n}\rho_i=t_{i,0}+\sum_{j=1}^{n-1} t_{ij}\omega_j\ \ (i=1\kdots n),
\]
where 
\begin{equation}\label{eq4.8}
t_{ij}=\mu\Delta'^{i-n}m_{ij}=\lambda\Delta^{j-n}m_{ij}^*\ \ (i,j=1\kdots n-1).
\end{equation}
Since $M=(m_{ij})\in\GL_{n-1}(A)$, the entries of each row of $M$ have gcd $1$.
It follows that the fractional ideal generated by the entries 
of the $i$-th row of $T=(t_{ij})_{i,j=1\kdots n-1}$ is
$[\mu\Delta'^{i-n}]$. 
Hence the fractional ideal generated by all entries of $T$ is 
$[\mu\Delta'^{1-n}]$. Similarly, since also $M^*\in\GL_{n-1}(A)$, 
the fractional ideal generated by the entries of the $j$-th column of $T$ is $[\lambda \Delta^{j-n}]$.
Hence the fractional ideal generated by all entries of $T$ is $[\lambda\Delta^{1-n}]$.
So $[\lambda\Delta^{1-n}]=[\mu\Delta'^{1-n}]$. On the other hand,
using $\det M\in A^*$, $\det M^*\in A^*$, we find
$[\det T]=[\lambda^{n-1}\Delta^{-n(n-1)/2}]=[\mu^{n-1}\Delta'^{-n(n-1)/2}]$.
By combining these two identities, using that $n\geq 4$, we obtain 
\[
[\lambda ]=[\mu ],\ \ [\Delta ]=[\Delta '].
\]
This proves \eqref{eq4.-1}.
Further, by \eqref{eq4.8},
\[
[m_{ij}]=[\Delta^{j-i}m_{ij}^*],
\]
and since $m_{ij}^*\in A$ this implies \eqref{eq4.9}.
Combining \eqref{eq4.9} with $\det M\in A^*$ we obtain \eqref{eq4.10}.

It remains to prove \eqref{eq4.11}.
Note that by \eqref{eq4.4} we have
\[
\omega_1^2=a_1\omega_1-a_0\omega_2,\ \ \rho_1^2=b_1\rho_1-b_0\rho_2.
\]
Substituting \eqref{eq4.5} and using the congruences \eqref{eq4.9}, we obtain
the following congruences modulo $\Delta A_{\alpha}$:
\begin{align*}
&b_1(m_{1,0}+m_{1,1}\omega_1)-b_0(m_{2,0}+m_{2,1}\omega_1+m_{2,2}\omega_2)
\\
&\equiv (m_{1,0}+m_{1,1}\omega_1 )^2\equiv m_{1,0}^2+2m_{1,0}m_{1,1}\omega_1+m_{1,1}^2\omega_1^2
\\
&\equiv m_{1,0}^2+2m_{1,0}m_{1,1}\omega_1+m_{1,1}^2(a_1\omega_1-a_0\omega_2)
\\
&\equiv m_{1,0}^2+(2m_{1,0}m_{1,1}+m_{1,1}^2a_1)\omega_1-m_{1,1}^2a_0\omega_2\ (\text{mod}\,\Delta A_{\alpha}).
\end{align*}
Comparing the coefficients of $\omega_2$, we see that $\mod{b_0m_{2,2}}{a_0m_{1,1}^2}{\Delta}$.
Combined with \eqref{eq4.10}, this gives \eqref{eq4.11}.
\end{proof}

For the remainder of the proof of \eqref{eq4.0} it will be convenient to work
locally. Let $\VV_L$ be the set of discrete valuations on $L$ corresponding 
to the non-zero prime ideals of $A_L$, i.e., $v\in\VV_L$ corresponds to the 
prime ideal $\fp$ if $v(x)$ is the exponent of $\fp$
in the unique prime ideal decomposition of $[x]$.
Further, put $\delta_v:=v(\fd )=\min\{ v(x):\, x\in\fd\}$.

\begin{lemma}\label{lem4.3}
Let $v\in\VV_L$. 
Then for all pairwise distinct $i,j,k,l\in\{ 1\kdots n\}$ we have
\begin{equation}\label{eq4.12}
v(\Delta )\leq 5\delta_v + 2\cdot v\Big(\frac{\cross_{ijkl}(\beta )}{\cross_{ijkl}(\alpha )}-1\Big).
\end{equation}
\end{lemma}

\begin{proof}
We assume without loss of generality
\begin{equation}\label{eq4.15}
v(\Delta )>5\delta_v.
\end{equation}
 
We frequently use the following facts. 
Let as before $x\mapsto x^{(i)}$ ($i=1\kdots n$) be the $\vk$-isomorphic embeddings of $K$ in $L$ so that
\[
f_{\alpha}=a_0(X-\alpha^{(1)})\cdots (X-\alpha^{(n)}),\ \ 
f_{\beta}=b_0(X-\beta^{(1)})\cdots (X-\beta^{(n)}).  
\]
Since $f_{\alpha},f_{\beta}$ are primitive,
we have by Gauss' Lemma,
\begin{equation}\label{eq4.13}
v(a_0)+\sum_{i=1}^n\min (0,v(\alpha^{(i)}))=0,\ \ v(b_0)+\sum_{i=1}^n\min (0,v(\beta^{(i)}))=0.
\end{equation}
By Lemma \ref{lem2.-1} we have $\fd =[D(f_{\alpha})]=[D({f_{\beta}})]$. 
Using $D(f_{\alpha})=\\ a_0^{2n-2}\prod_{1\leq i<j\leq n} (\alpha^{(i)}-\alpha^{(j)})^2$
and likewise for $f_{\beta}$, and inserting \eqref{eq4.13}, we obtain
\begin{equation}\label{eq4.14}
\begin{split}
\half\delta_v &=\sum_{1\leq i<j\leq n}\!\! \Big(v(\alpha^{(i)}-\alpha^{(j)})-\min (0,v(\alpha^{(i)}))-\min (0,v(\alpha^{(j)}))\Big)
\\
&=\sum_{1\leq i<j\leq n}\!\! \Big(v(\beta^{(i)}-\beta^{(j)})-\min (0,v(\beta^{(i)}))-\min (0,v(\beta^{(j)}))\Big).
\end{split}
\end{equation}
For $a,b,c\in L$ we write $\mod{a}{b}{c}$ 
if $v(a-b)\geq v(c)$.
By \eqref{eq4.5} and \eqref{eq4.9} we have
\[ 
\mod{b_0\beta^{(i)}}{m_{1,0}+m_{1,1}a_0\alpha^{(i)}}{\Delta}\ \ \text{for } i=1\kdots n;
\]
here we used that $\omega_j,\, \rho_j$ $(j=1\kdots n-1)$ and their
conjugates all lie in $A_L$. This implies
\begin{equation}\label{eq4.14a}
\mod{b_0(\beta^{(i)}-\beta^{(j)})}{m_{1,1}a_0(\alpha^{(i)}-\alpha^{(j)})}{\Delta}\ \ 
\text{for } i,j=1\kdots n.
\end{equation}

In the remainder of the proof we distinguish the two cases $v(a_0)\leq\half v(\Delta )$                                                                                                                                                                                                                           and $v(a_0)>\half v(\Delta )$. First assume that 
\[
v(a_0)\leq \half v(\Delta ).
\]
Let $i,j$ be any two distinct indices from $\{ 1\kdots n\}$.
Then by \eqref{eq4.10}, \eqref{eq4.14},
\begin{align*}
&v(m_{1,1}a_0(\alpha^{(i}-\alpha^{(j)}))
\\
&\qquad\leq v(a_0)+v(\alpha^{(i)}-\alpha^{(j)})-
\min (0,v(\alpha^{(i)}))-\min (0,v(\alpha^{(j)}))
\\
&\qquad\leq \half v(\Delta )+\half\delta_v,
\end{align*}
and together with \eqref{eq4.14a} this gives
\[
v\Big(\frac{b_0(\beta^{(i)}-\beta^{(j)})}{m_{1,1}a_0(\alpha^{(i)}-\alpha^{(j)})}-1\Big)
\geq \half v(\Delta )-\half \delta_v,
\]
which is $>0$ by \eqref{eq4.15}.
Using the trivial observation for discrete valuations
\begin{equation}\label{eq4.observation}
v(x_i-1)\geq c>0\ \text{for } i=1,2,3,4 \Longrightarrow v\Big(\frac{x_1x_2}{x_3x_4}-1\Big)\geq c
\end{equation}
we deduce
\[
v\Big(\frac{\cross_{ijkl}(\beta )}{\cross_{ijkl}(\alpha )}-1\Big)\geq \half v(\Delta )-\half\delta_v
\]
for all pairwise distinct $i,j,k,l\in\{ 1\kdots n\}$, which implies \eqref{eq4.12}.
\\[0.2cm]

Next, assume that
\begin{equation}\label{eq4.a0}
v(a_0)>\half v(\Delta ).
\end{equation} 
Then \eqref{eq4.11} implies that also 
\begin{equation}\label{eq4.b0}
v(b_0)>\half v(\Delta ).
\end{equation}

We first observe
\begin{equation}\label{eq4.22}
v(a_1)\leq \delta_v,\ \ v(b_1)\leq\delta_v.
\end{equation}
Indeed, recall that the discriminant $D(F)$ of a binary form
$F=\\ \sum_{i=0}^nx_iX^{n-i}Y^i$ is a polynomial in $\Zz [x_0\kdots x_n]$.
Consequently, if $F$ as above and 
$G=\sum_{i=0}^ny_iX^{n-i}Y^i$ are binary forms in $A[X,Y]$, we have $v(D(F)-D(G))\geq\min_{0\leq i\leq n} v(x_i-y_i)$.
Applying this with  $F(X,Y)=Y^nf_{\alpha}(X/Y)$ and $G(X,Y)=F(X,Y)-a_0X^n-a_1X^{n-1}Y$,
and noting that $D(G)=0$ since $G$ is divisible by $Y^2$,
we have
\[
\delta_v=v(D(F))=v(D(F)-D(G))\geq\min (v(a_0),v(a_1)).
\]
By \eqref{eq4.a0}, \eqref{eq4.15} we have $v(a_0)>\half v(\Delta )>\delta_v$.
Hence $v(a_1)\leq\delta_v$. The proof of $v(b_1)\leq\delta_v$ is the same, using \eqref{eq4.b0}
instead of \eqref{eq4.a0}.
 
Assume without loss of generality that
\[
v(\alpha^{(1)})=\min (v(\alpha^{(1)})\kdots v(\alpha^{(n)})).
\]
Then
\begin{equation}\label{eq4.18a}
v(\alpha^{(i)})\geq -\half \delta_v\ \ \text{for } i=2\kdots n.
\end{equation}
Indeed, suppose that $v(\alpha^{(i)})<-\half\delta_v$ for some $i\geq 2$. Then 
by \eqref{eq4.14},
\[
\half\delta_v\geq v\big({\alpha^{(1)}}^{-1}-{\alpha^{(i)}}^{-1}\big)>\half\delta_v,
\]
which is impossible. Thus,
\begin{equation}\label{eq4.18}
\begin{split}
v(a_0\alpha^{(i)}) &>\half v(\Delta )-\half \delta_v
\ \ \text{for } i\geq 2,
\\
v(a_0\alpha^{(1)})&=v(a_1+a_0(\alpha^{(2)}+\cdots +\alpha^{(n)}))\leq \delta_v,
\end{split}
\end{equation}
where in the derivation of the first inequality we used \eqref{eq4.a0}
and in that of the last inequality \eqref{eq4.22}, \eqref{eq4.15}.

Let $k$ be an index such that 
\[
v(\beta^{(k)})=\min (v(\beta^{(1)})\kdots v(\beta^{(n)})).
\]
Then completely similarly to
\eqref{eq4.18a}, \eqref{eq4.18} we derive
\begin{equation}\label{eq4.19}
\begin{split}
&v(\beta^{(i)})\geq -\half\delta_v\ \ \text{for } i\not= k,
\\
&v(b_0\beta^{(k)}) \leq\delta_v,\
v(b_0\beta^{(i)}) >\half v(\Delta )-\half\delta_v\ \ \text{for }i\not= k,
\end{split}
\end{equation}
where we used \eqref{eq4.b0} instead of \eqref{eq4.a0}.
We show that the index $k$ must be equal to $1$.
Recall that by \eqref{eq4.14a},
\[
\mod{b_0(\beta^{(i)}-\beta^{(1)})}{m_{1,1}a_0(\alpha^{(i)}-\alpha^{(1)})}{\Delta}\ \ 
\text{for } i\geq 2.
\]
Assuming $k\not= 1$,  for $i\not= 1,k$  this congruence contradicts the two inequalities
\begin{align*}
&v(b_0(\beta^{(i)}-\beta^{(1)}))>\half v(\Delta )-\half\delta_v\ \ 
\text{implied by \eqref{eq4.19},}
\\
&v(m_{1,1}a_0(\alpha^{(i)}-\alpha^{(1)}))\leq \delta_v<\half v(\Delta )-\half\delta_v
\end{align*}
implied by \eqref{eq4.10}, \eqref{eq4.18}, \eqref{eq4.15}.
So indeed $k=1$, and thus, \eqref{eq4.19} becomes
\begin{equation}\label{eq4.21}
\begin{split}
&v(\beta^{(i)})\geq -\half\delta_v\ \ \text{for } i\geq 2,
\\
&v(b_0\beta^{(1)}) \leq \delta_v,\
v(b_0\beta^{(i)}) >\half v(\Delta )-\half\delta_v\ \ \text{for }i\geq 2.
\end{split}
\end{equation}

Let $i\in\{ 2\kdots n\}$. By \eqref{eq4.5} and \eqref{eq4.9} we have
\[
\mod{b_0{\beta^{(i)}}^2+b_1\beta^{(i)}}{m_{2,0}+m_{2,1}a_0\alpha^{(i)}+m_{2,2}(a_0{\alpha^{(i)}}^2+a_1\alpha^{(i)})}{\Delta},
\]
while
\begin{align*}
&v(a_0({\alpha^{(i)}})^2)> \half v(\Delta )-\delta_v\ \ 
\text{by \eqref{eq4.a0}, \eqref{eq4.18a},}
\\
&v(b_0({\beta^{(i)}})^2)>\half v(\Delta )-\delta_v\ \ 
\text{by \eqref{eq4.b0}, \eqref{eq4.21},}
\\
&v(a_0\alpha^{(i)})>\half v(\Delta )-\half \delta_v\ \ \text{by \eqref{eq4.a0}, \eqref{eq4.18a}.}
\end{align*}
These relations together imply
\[
v(b_1\beta^{(i)}-m_{2,0}-m_{2,2}a_1\alpha^{(i)})>\half v(\Delta )-\delta_v\ \ 
\text{for } i\geq 2.
\]

Now let $i,j$ be any two distinct indices with $2\leq i,j\leq n$. Then by the inequality just derived,
\begin{equation}\label{eq4.24}
v\big(b_1(\beta^{(i)}-\beta^{(j)})-m_{2,2}a_1(\alpha^{(i)}-\alpha^{(j)})\big)>
\half v(\Delta ) -\delta_v .
\end{equation}
Further, by \eqref{eq4.10}, \eqref{eq4.22},
\eqref{eq4.14},
\[
\begin{split}
&v(m_{2,2}a_1(\alpha^{(i)}-\alpha^{(j)})) 
\\
&\qquad \leq v(a_1)+v(\alpha^{(i)}-\alpha^{(j)})-\min (0,v(\alpha^{(i)}))-\min (0,v(\alpha^{(j)}))
\\
&\qquad \leq\smallfrac{3}{2}\delta_v,
\end{split}
\]
which together with \eqref{eq4.24} implies
\begin{equation}\label{eq4.25}
v\Big(\frac{b_1(\beta^{(i)}-\beta^{(j)})}{m_{2,2}a_1(\alpha^{(i)}-\alpha^{(j)})}-1\Big)>
\half v(\Delta )-\smallfrac{5}{2}\delta_v.
\end{equation}

Inequality \eqref{eq4.25} holds for any pair of indices $i,j\geq 2$.
We still have to look at the case where one of the indices is $1$.
Let $j\geq 2$. Then by \eqref{eq4.18}, \eqref{eq4.15},
\[
v(a_0(\alpha^{(1)}-\alpha^{(j)}))\leq \delta_v,
\]
which together with \eqref{eq4.14a} implies
\begin{equation}\label{eq4.28}
v\Big(\frac{b_0(\beta^{(1)}-\beta^{(j)})}{m_{1,1}a_0(\alpha^{(1)}-\alpha^{(j)})}-1\Big)
>v(\Delta )-\delta_v.
\end{equation}
Finally, from \eqref{eq4.25}, \eqref{eq4.28}, \eqref{eq4.15} and observation \eqref{eq4.observation}
we deduce
\[
v\Big(\frac{\cross_{ijkl}(\beta )}{\cross_{ijkl}(\alpha )}-1\Big)>
\half v(\Delta )-\smallfrac{5}{2}\delta_v
\]
for all pairwise distinct $i,j,k,l\in\{ 1\kdots n\}$.
This implies \eqref{eq4.12} and thus completes the proof of Lemma \ref{lem4.3}.
\end{proof}

\begin{proof}[Proof of Proposition \ref{prop4.1}]
By applying Lemma \ref{lem4.3} for all $v\in\VV_L$, \eqref{eq4.0} clearly follows.
\end{proof}

\section{Proof of Theorems \ref{thm1.1}--\ref{thm1.3}}\label{section5}

Let $K$ be a number field.
Recall that $\alpha_1\in K$ is called $k$-special if $K=\Qq (\alpha_1 )$
and there are $\alpha_2\kdots \alpha_k$
such that $\alpha_1\kdots\alpha_k$ are pairwise $\GL_2(\Zz )$-inequivalent and 
$\Zz_{\alpha_1}=\cdots =\Zz_{\alpha_k}$. A $2$-special number in $K$ is called special.
We first prove the following proposition.

\begin{proposition}\label{prop5.1}
Let $K$ be a number field of degree $n\geq 3$.
Then every $\GL_2(\Qq )$-equivalence class of special
$\alpha\in K$ is the union of at most finitely many $\GL_2(\Zz )$-equivalence classes.
\end{proposition}

\begin{proof}
First let $n=3$. By Lemma \ref{lem2.3} (i) and Lemma \ref{lem2.4},
any two numbers $\alpha$, $\beta$ with $\Zz_{\alpha}=\Zz_{\beta}$
are $\GL_2(\Zz )$-equivalent. Hence there are no special numbers in $K$.

Next let $n\geq 4$. Denote by $L$ the normal closure of $K$.
Let $\CC$ be a $\GL_2(\Qq )$-equivalence class of special $\alpha\in K$.
We first split $\CC$ into a finite collection of subclasses.
Since cross ratios of $\GL_2(\Qq )$-equivalent numbers are the same, we may define
$\cross_{ijkl}(\CC ):=\cross_{ijkl}(\alpha )$ for any $\alpha\in\CC$
and any four distinct indices $i,j,k,l\in\{ 1\kdots n\}$.
For every $\alpha\in\CC$ there is $\beta\in K$ 
such that $\Zz_{\alpha}=\Zz_{\beta}$ 
and $\beta$ is not $\GL_2 (\Zz )$-equivalent to $\alpha$.  
From Lemma \ref{lem2.2} and \eqref{eq3.3} it follows that
$\ve_{ijkl}:=\cross_{ijkl}(\beta )/\cross_{ijkl}(\alpha )\in\OO_L^*$ for all distinct $i,j,k,l\in\{ 1\kdots n\}$ and
\begin{equation}\label{eq5.0}
\cross_{ijkl}(\CC )\ve_{ijkl}+\cross_{ilkj}(\CC )\ve_{ilkj}=1
\end{equation}
for all distinct $i,j,k,l\in\{ 1\kdots n\}$. 

We apply the following result, due to Lang \cite{L60}.

\begin{lemma}\label{lem5.2}
Let $F$ be a field of characteristic $0$, let $a,b\in F^*$, 
and let $\Gamma$ be a subgroup of $F^*$ of finite rank. Then the equation
\[
ax+by=1\ \ \text{in } x,y\in\Gamma
\]
has only finitely many solutions.
\end{lemma}

By applying this to \eqref{eq5.0} with $\Gamma =\OO_L^*$, we infer that there is a finite set depending only on $\CC$ such that for all $i,j,k,l$,
$\ve_{ijkl}$ belongs to this set, and so, for all $i,j,k,l$, 
$\cross_{ijkl}(\beta )$
belongs to a finite set depending only on $\CC$.
Now Lemma \ref{lem2.3} (ii)
implies that the $\GL_2(\Qq )$-equivalence class of $\beta$ belongs to a finite collection
depending only on $\CC$. Further, by Lemma \ref{lem2.4}, the classes in this
collection are disjoint from $\CC$. This implies that $\CC$ can be partitioned
into a finite collection of subclasses 
\[
\CC (\DD ):=\{ \alpha\in\CC:\ \text{there is $\beta\in\DD$ with } \Zz_{\alpha}=\Zz_{\beta}\},
\]
where $\DD$ is a $\GL_2(\Qq )$-equivalence class of special numbers distinct from $\CC$.

Take a $\GL_2(\Qq )$-equivalence class $\DD\not=\CC$ for which $\CC (\DD )\not=\emptyset$.
We have to show that $\CC (\DD )$ is the union of finitely many
$\GL_2(\Zz )$-equivalence classes.
We use that for every positive integer $\Delta$ there is a finite set of integer
$2\times 2$-matrices $\FF (\Delta )$, such that if $C$ is any $2\times 2$-matrix with 
$|\det C|=\Delta$, 
then there is $U\in\GL_2(\Zz )$ with $UC\in \FF (\Delta )$.

Fix $\alpha\in\CC (\DD )$ and then $\beta\in\DD$ with $\Zz_{\alpha}=\Zz_{\beta}$.
Then choose $\alpha^*\in\CC (\DD )$; we let $\alpha^*$ vary. Further choose
$\beta^*\in\DD$ with $\Zz_{\alpha^*}=\Zz_{\beta^*}$.
Thus, $(\alpha ,\beta )$ and $(\alpha^*, \beta^*)$ are
two $\GL_2(\Qq )$-equivalent special pairs as in Proposition \ref{prop4.1}, with $A=\Zz$.
Let $C$ be the matrix from \eqref{eq4.1}, so with $\alpha^*=C\alpha$,
and put $\Delta :=|\det C |$.
Then there is $U\in\GL_2(\Zz )$ such that
\[
UC=:C_1\in \FF (\Delta ).
\]
Let $\alpha^{**}:=U\alpha^*=C_1\alpha$. 
By Proposition \ref{prop4.1}, $\Delta$ belongs to a finite set depending on $\alpha ,\beta$,
hence so does $C_1$, and thus $\alpha^{**}$.
This implies that the $\GL_2(\Zz )$-equivalence class of $\alpha^*$  belongs to a finite 
collection depending on $\alpha ,\beta$. 
This shows that indeed, $\CC (\DD )$ is the union of finitely many 
$\GL_2(\Zz )$-equivalence classes.
\end{proof}

\begin{proof}[Proof of Theorem \ref{thm1.1}]
Propositions \ref{prop3.1} and \ref{prop5.1} imply that if $K$
is quartic then the $3$-special numbers $\alpha\in K$ lie in finitely many
$\GL_2(\Zz )$-equivalence classes. Further, if $K$ has degree $\geq 5$
and the Galois group of its normal closure is $5$-transitive,
then the special numbers in $K$ lie in finitely many $\GL_2(\Zz )$-equivalence classes.
As we observed in Section \ref{section3}, this implies Theorem \ref{thm1.1}.
\end{proof}

\begin{proof}[Proof of Theorem \ref{thm1.2}]
Let $K$ be either a quartic field, or a number field of degree $\geq 5$ such that the Galois group of the normal closure
of $K$ is $5$-transitive. Consider a Hermite equivalence class $\HH$
of polynomials in $\PP\II (K)$  that falls apart into at least three $\GL_2(\Zz )$-equivalence classes if $[K:\Qq ]=4$, and into at least two $\GL_2(\Zz )$-equivalence classes
if $[K:\Qq ]\geq 5$.
Recall that two polynomials $f,g\in\PP\II (K)$ are Hermite equivalent
if $f$ has a root $\alpha$ and $g$ a root $\beta$ such that $\Qq (\alpha )=\Qq (\beta)=K$ and 
$\MM_{\beta }=\lambda\MM_{\alpha}$
for some non-zero $\lambda$. This implies $\Zz_{\alpha}=\Zz_{\beta}$.
Now if $f,g\in\HH$ are $\GL_2(\Zz )$-inequivalent, then so are $\alpha ,\beta$.
So the order $\OO=\Zz_{\alpha}$ has at least three rational monogenizations
if $[K:\Qq ]=4$, and at least two rational monogenizations if $[K:\Qq ]\geq 5$. 
Since $\OO$ is an order
of a conjugate of $K$ and $K$ has only finitely many conjugates,
Theorem \ref{thm1.1} implies that there are only finitely many 
possibilities for $\OO$. 
Given $\OO$, the set of $\alpha$ with $\Zz_{\alpha}=\OO$
is the union of finitely many $\GL_2(\Zz )$-equivalence classes.
Hence the set of $f\in\PP\II (K)$ having a root $\alpha$ with $\Zz_{\alpha}=\OO$
is the union of finitely many $\GL_2(\Zz )$-equivalence classes.
The class $\HH$ is the union of some of these classes.  
So we have only finitely many possibilities for $\HH$.
\end{proof}

\begin{proof}[Proof of Theorem \ref{thm1.3}]
Take an algebraic number $\alpha$ of degree $n\geq 3$.
Let $f_{\alpha}(X)=a_0X^n+\cdots +a_n\in\Zz [X]$ be the primitive minimal polynomial
of $\alpha$ and $F_{\alpha}(X,Y):=X^nf_{\alpha} (X/Y)$ its homogenization.
By Thue's Theorem \cite{T1909}, there is a number $C$ such that if $x,y$ are integers
with $F_{\alpha}(x,y)=\pm 1$, then $|x|,|y|\leq C$.
Let $p,q$ be distinct prime numbers such that $p,q>C^*:=\max (C, |a_0|,|a_n|)$.
The number $(q/p)\alpha$ has primitive minimal polynomial $f_{q\alpha /p}(X)=q^nf_{\alpha}(pX/q)$
(one verifies easily that the coefficients of this polynomial have $\gcd$ $1$,
since $p,q >|a_0|,|a_n|)$. The polynomial $f_{q\alpha /p}$, hence by \eqref{eq1.4} the order
$\Zz_{q\alpha /p}$, has discriminant $(pq)^{n(n-1)}D(f_{\alpha})$.
So the orders $\Zz_{q\alpha /p}$, with $p,q$ running through the primes exceeding $C^*$,
are all different.

We claim that among these orders, at most finitely many are monogenic. 
Indeed, suppose that $\Zz_{q\alpha /p}$ is monogenic.
Then $\Zz_{q\alpha/p}=\Zz_{\beta}=\Zz [\beta ]$ for some algebraic integer $\beta$.
Assume that $\beta$ is $\GL_2(\Zz )$-equivalent to $q\alpha /p$. That is,
$\beta =\frac{a(q\alpha /p)+b}{c(q\alpha /p)+d}$ for some $\big(\begin{smallmatrix}a&b\\ c&d\end{smallmatrix}\big)\in\GL_2(\Zz )$. Then the necessarily monic primitive minimal polynomial
of $\beta$ is 
\[
f_{\beta}(X)=\pm q^n(-cX+a)^nf_{\alpha}\Big(\frac{p(dX-b)}{q(-cX+a)}\Big).
\]
Its homogenization is 
\[
F_{\beta}(X,Y)=Y^nf_{\beta}(X/Y)=\pm F_{\alpha}(p(dX-bY),q(-cX+aY)).
\]
Since $\beta$ is integral, the leading coefficient of $f_{\beta}$ is $1$,
which implies $1 =F_{\beta}(1,0)=\pm F_{\alpha}(pd, -qc)$. But this is impossible,
since at least one of $|pd|,|qc|$ exceeds the bound $C$ defined above. We conclude that
$\beta$ cannot be $\GL_2(\Zz )$-equivalent to $q\alpha /p$. So any order $\Zz_{q\alpha /p}$
that is monogenic must have two rational monogenizations. By Proposition \ref{prop5.1} 
there are at most finitely many pairs of distinct primes $p,q>C^*$ for which this is possible.
This leaves us with infinitely many rationally monogenic orders $\Zz_{q\alpha /p}$
that are not monogenic.
\end{proof}

\section{A generalization over the $S$-integers}\label{section6}

In this section, we will state and prove a generalization of Theorem \ref{thm1.1}
to the ring $\OO_S$ of $S$-integers of a number field.
The ring of $S$-integers is a Dedekind domain, but in general not a principal ideal domain,
therefore, the arguments from the previous sections cannot be carried over.
Thus, in our generalization of Theorem \ref{thm1.1} we will not work
with $\GL_2(\OO_S)$-equivalence of algebraic numbers,
but rather with numbers that are $\GL_2(\OO_{\fp})$-equivalent for all
non-zero prime ideals $\fp$ of $\OO_S$, where $\OO_{\fp}$
is the localization of $\OO_S$ at $\fp$. 

Before stating and proving our result, we have collected some generalizations
of the material from Section \ref{section2} to Dedekind domains of characteristic $0$.
Most of these are equivalent, but for our purposes more convenient
formulations of material from \cite[Chap. 17]{EG17}.

Let $A$ be a Dedekind domain of characteristic $0$ and $\vk$ its quotient field.
Denote by $\PP (A)$ the collection of non-zero prime ideals of $A$ and by $Cl (A)$
the class group of $A$ (fractional ideals modulo principal fractional ideals).
Further, let $Cl (A)[m]$ be the subgroup of elements of $Cl (A)$ whose $m$-th power
is the principal ideal class. 
The localization of $A$ at a prime ideal $\fp\in\PP (A)$
is given by 
\[
A_{\fp}:=\{ x/y:\, x\in A, y\in A\setminus\fp\}.
\]
We define the group of matrices
\[
G(A):= \bigcap_{\fp\in\PP (A)} \vk^*\GL_2(A_{\fp}),
\]
that is the group of matrices $C$ such that for every $\fp\in\PP (A)$ there is
$\lambda_{\fp}\in\vk^*$ with $\lambda_{\fp}^{-1}C\in\GL_2(A_{\fp})$.

Let $\alpha ,\beta\in\overline{\vk}$ be of degree $\geq 3$ over $\vk$.
We say that $\alpha ,\beta$ are $G(A)$-\emph{equivalent} if there is $C\in G(A)$
with $\beta =C\alpha$. Then
\begin{align}\label{eq6.2}
&\alpha ,\beta\ \,\text{are $G(A)$-equivalent}
\\[-0.1cm]
\notag
&\qquad\Longleftrightarrow \alpha ,\beta\ \,
\text{are $\GL_2(A_{\fp})$-equivalent for every $\fp\in\PP(A)$}.
\end{align}
Indeed, $\Rightarrow$ is clear. As for $\Leftarrow$, suppose that 
$\alpha ,\beta$ are $\GL_2(A_{\fp})$-equivalent for every $\fp\in\PP (A)$.
Then there is $C\in\GL_2(\vk )$ such that $\beta =C\alpha$. But $C$ is determined
uniquely up to a scalar in $\vk^*$, hence $C\in\vk^*\GL_2(A_{\fp})$
for every $\fp\in \PP (A)$, i.e., $C\in G(A)$.

We compare $G(A)$-equivalence with $\GL_2(A)$-equivalence. 

\begin{lemma}\label{lem6.1}
$G(A)/\vk^*\GL_2(A) \cong Cl(A)[2]$.
\end{lemma}

\begin{proof} 
Let $[a_1\kdots a_r]$ denote the fractional ideal of $A$ generated by $a_1\kdots a_r$
and for a matrix $C$ with entries in $\vk$,
let $[C]$ denote the fractional ideal generated by the entries of $C$. We claim that
\begin{equation}\label{eq6.3}
G(A)=\{ C\in\GL_2(\vk ):\, [\det C]=[C]^2\}.
\end{equation}
Indeed, let $C\in G(A)$. 
Then for all $\fp\in\PP (A)$ there is $\lambda_{\fp}\in \vk^*$
such that $\lambda_{\fp}^{-1}C\in \GL_2(A_{\fp})$, hence
$[C]^2\cdot A_{\fp}=\lambda_{\fp}^2A_{\fp}=[\det C]\cdot A_{\fp}$ for all $\fp$,
implying $[C]^2=[\det C]$.
Conversely, assume $[\det C]=[C]^2$. Then for all $\fp\in\PP (A)$ 
there is $\lambda_{\fp}\in\vk^*$ with $[C]A_{\fp}=\lambda_{\fp}A_{\fp}$
since $A_{\fp}$ is a principal ideal domain. So $\det (\lambda_{\fp}^{-1}C)=\lambda_{\fp}^{-2}\det C\in A_{\fp}^*$, i.e., $\lambda_{\fp}^{-1}C\in\GL_2(A_{\fp})$ for all $\fp\in\PP (A)$,
implying $C\in G(A)$.

Now define the map
\[
G(A) \to Cl(A)[2]:\ \ C\mapsto\ \text{ideal class of $[C]$}.
\]
By \eqref{eq6.3} this is a well-defined group homomorphism. The kernel of this homomorphism
is the group of matrices $C\in G(A)$ such that $[C]$ is principal,
this is  precisely $\vk^*\GL_2(A)$. To show that the homomorphism is surjective,
pick any ideal class of $A$ whose square is principal,
and take an ideal from this class. 
By a well-known property of Dedekind domains, this ideal is generated by two elements, say it is $[a,b]$.
Then, using another property of Dedekind domains, $[a^2,b^2]=[a,b]^2=[\lambda ]$ for some $\lambda\in A$, hence there are $u,v\in A$ such that $ua^2-vb^2=\lambda$.
Take $C=\big(\begin{smallmatrix}a&b\\ vb&ua\end{smallmatrix}\big)$.
Then $[C]^2=[a,b]^2=[\lambda ]=[\det C]$, so $C\in G(A)$, and $C$
maps to the ideal class of $[a,b]$.
\end{proof}

Lemma \ref{lem6.1} implies that a $G(A)$-equivalence class
is the union of precisely $\#(Cl(A)[2])$ $\GL_2(A)$-equivalence classes.
This quantity is finite for instance if $A$ is the ring of $S$-integers
of a number field.

Let $K$ be a finite extension of $\vk$ of degree $n\geq 3$.
Given $\alpha$ with $\vk (\alpha )=K$, we define the $A$-module
\[
\MM_{\alpha}:=\{ x_0+x_1\alpha +\cdots +x_{n-1}\alpha^{n-1}:\, x_0\kdots x_{n-1}\in A\}
\]
and its ring of scalars
\[
A_{\alpha}:=\{ \xi\in K:\, \xi\MM_{\alpha}=\MM_{\alpha}\}.
\]
For $\fp\in \PP (A)$, let $\MM_{\fp ,\alpha}$ be the $A_{\fp}$-module generated
by $1,\alpha\kdots \alpha^{n-1}$ and 
$A_{\fp ,\alpha}:=\{\xi\in K:\ \xi\MM_{\fp ,\alpha}\subseteq\MM_{\fp ,\alpha}\}$.
Then
\begin{align}\label{eq6.4}
&A_{\fp ,\alpha}=A_{\fp}A_{\alpha}\ \ \text{for all } \fp\in \PP (A),
\\
\label{eq6.5}
&\displaystyle{A_{\alpha}=\bigcap_{\fp\in\PP (A)} A_{\fp ,\alpha}.}
\end{align}

\begin{lemma}\label{lem6.2}
Let $\alpha ,\beta\in K$ such that $\vk (\alpha )=\vk (\beta )=K$ and $\alpha ,\beta$ are
$G(A)$-equivalent. Then $A_{\alpha}=A_{\beta}$.
\end{lemma}

\begin{proof}
From \eqref{eq6.2} it follows that $\alpha ,\beta$ are $\GL_2(A_{\fp})$-equivalent
for all $\fp$, so
$A_{\fp ,\alpha}=A_{\fp ,\beta}$ for all $\fp$. Now apply \eqref{eq6.5}.
\end{proof}

\begin{lemma}\label{lem6.4}
Let $\alpha ,\beta\in K$ such that $\vk (\alpha )=\vk (\beta )=K$ and $A_{\alpha}=A_{\beta}$.
Suppose that $\alpha ,\beta$ are $\GL_2(\vk )$-equivalent.
Then they are $G(A)$-equivalent.
\end{lemma}

\begin{proof}
From \eqref{eq6.4} it follows that $A_{\fp ,\alpha}=A_{\fp ,\beta}$ and then from Lemma \ref{lem2.4} that $\alpha ,\beta$ are $\GL_2(A_{\fp})$-equivalent
for all $\fp\in\PP (A)$; here we have used that the $A_{\fp}$ are principal ideal domains.
Now \eqref{eq6.2} implies that they are $G(A)$-equivalent.
\end{proof}

Suppose that $[K:\vk ]=n\geq 4$.
Let $L$ be the normal closure of $K/\vk$ and $x\mapsto x^{(i)}$ ($i=1\kdots n$)
the $\vk$-isomorphic embeddings $K\hookrightarrow L$.
Denote by $A_L$ the integral closure of $A$ in $L$.
Define the cross ratios $\cross_{ijkl}(\alpha )$ ($K=\vk (\alpha )$) by \eqref{eq4.crossratio}.

\begin{lemma}\label{lem6.5}
Let $\alpha$, $\beta$ be such that $\vk (\alpha )=\vk (\beta )=K$ and $A_{\alpha}=A_{\beta}$.
Then for all pairwise distinct $i,j,k,l\in\{ 1\kdots n\}$ we have 
\[
\frac{\cross_{ijkl}(\alpha )}{\cross_{ijkl}(\beta )}\in A_L^*.
\]
\end{lemma}

\begin{proof}
For $\fp\in \PP (A)$, let $A_{\fp ,L}$ be the integral closure of $A_{\fp}$ in $L$.
Then $\cap_{\fp\in \PP (A)} A_{\fp ,L}=A_L$. 
By \eqref{eq6.4} we have $A_{\fp ,\alpha}=A_{\fp ,\beta}$, and so by Lemma \ref{lem2.2},
$\frac{\cross_{ijkl}(\alpha )}{\cross_{ijkl}(\beta )}\in A_{\fp ,L}^*$ for all $\fp\in\PP (A)$.
Since $\cap_{\fp \in\PP (A)} A_{\fp ,L}^*=A_L^*$ this implies our lemma.
\end{proof}

We now specialize to rings of $S$-integers of number fields.
Let $\vk$ be a number field and $\OO_{\vk}$ its ring of integers.
Let $S$ be a finite set of non-zero prime ideals of $\OO_{\vk}$, 
and 
\[
\OO_S :=\{ x/y:\, x,y\in\OO_{\vk},\ y\ \text{composed of prime ideals from } S\}
\]
the ring of $S$-integers. Similarly as before, we denote by $\PP (\OO_S)$
the set of non-zero prime ideals of $\OO_S$. Further, for $\fp\in\PP (\OO_S)$,
we denote by $\OO_{\fp}$ the localization of $\OO_S$ at $\fp$, so that 
\[
G(\OO_S)=\bigcap_{\fp\in\PP (\OO_S)} \vk^*\GL_2(\OO_{\fp}).
\]

Let $K$ be a finite extension of $\vk$ of degree $n\geq 4$, and $L$ the
normal closure of $K/\vk$. 

Denote by $\OO_{S,K}$ the integral closure of $\OO_S$ in $K$.
By an $\OO_S$-\emph{order} of $K$ we mean a ring $\OO$ such that 
$\OO_S\subseteq\OO\subseteq\OO_{S,K}$ and $\vk\OO =K$. 

Recall that $\alpha ,\beta\in K$ are called $G(\OO_S)$-equivalent if
$\beta =C\alpha$ for some $C\in G(\OO_S)$. 
A \emph{rational monogenization} of an $\OO_S$-order $\OO$ 
is a $G(\OO_S)$-equivalence class of $\alpha$
such that $\OO_{S,\alpha}=\OO$.

Taking $\alpha$ with $K=\vk (\alpha )$, we say that the Galois group ${\rm Gal}(L/\vk )$
is $t$-transitive if the action of ${\rm Gal}(L/\vk )$ on the set of conjugates of $\alpha$ in $L$ 
is $t$-transitive.

We are now ready to state our generalization.

\begin{theorem}\label{thm6.6}
Let $\vk$ be an algebraic number field and $S$ a finite set of prime ideals from $\OO_{\vk}$.
Further, let $K$ be a finite extension of $\vk$, and $L$ the normal closure of $K/\vk$.
\\[0.1cm]
(i) Assume that $[K:\vk ]=4$. Then $K$ has only finitely many $\OO_S$-orders with more than two rational monogenizations.
\\[0.1cm]
(ii) Assume that $[K:\vk ]\geq 5$ and that ${\rm Gal}(L/\vk )$ is $5$-transitive.
Then $K$ has only finitely many $\OO_S$-orders with more than one rational monogenization.
\end{theorem}

The proof is very similar to that of Theorem \ref{thm1.1}. We will mainly focus
on the differences.

We keep the notation and assumptions from Theorem 6.5. 
We call $\alpha_1\in K$ $k$-special
if $\vk (\alpha_1)=K$ and if there are $\alpha_2\kdots \alpha_k\in K$
such that $\alpha_1\kdots \alpha_k$ are pairwise $G(\OO_S)$-inequivalent
and $\OO_{S,\alpha_1}=\cdots =\OO_{S,\alpha_k}$. We call $\alpha_1$ special if it is
$2$-special.

\begin{proof}[Proof of Theorem \ref{thm6.6}]
It suffices to show that if $[K:\vk ]=4$ then the $3$-special numbers in $K$
lie in at most finitely many $G (\OO_S)$-equivalence classes, while if $[K:\vk ]\geq 5$
and ${\rm Gal}(L/\vk )$ is $5$-transitive then the special numbers in $K$
lie in at most finitely many $G(\OO_S)$-equivalence classes.
\\[0.2cm]
{\bf Step 1.} \emph{The $3$-special numbers in $K$ if $[K:\vk ]=4$,
respectively the special numbers in $K$ if $[K:\vk ]\geq 5$ lie in at most finitely many 
$\GL_2(\vk )$-equivalence classes.}
\\[0.1cm]
The proof is exactly the same as that of Proposition \ref{prop3.1},
replacing everywhere $\Zz$, $\Qq$, $\OO_L^*$ by $\OO_S$, $\vk$, $\OO_{S,L}^*$,
where $\OO_{S,L}$ is the integral closure of $\OO_S$ in $L$.
Lemmas \ref{lem3.2} and \ref{lem3.3} can be applied with $\Gamma =\OO_{S,L}^*$,
since the latter group is finitely generated by the Dirichlet-Chevalley-Weil theorem.
\\[0.2cm]
{\bf Step 2.} \emph{Let $K$ be any extension of $\vk$ with $[K:\vk ]\geq 4$.
Then each $\GL_2(\vk )$-equivalence class of special numbers in $K$
is the union of finitely many $G(\OO_S)$-equivalence classes.}
\\[0.1cm]
Let $\CC$ be a $\GL_2(\vk )$-equivalence class of special numbers in $K$.
Completely similarly as in the proof of Proposition \ref{prop5.1},
applying Lemma \ref{lem6.5}, Lemma \ref{lem5.2} with $\Gamma =\OO_{S,L}^*$,
and Lemma \ref{lem6.4},
one shows that $\CC$ is the union of finitely many subclasses
\[
\CC (\DD ):=\{ \alpha \in\CC:\ 
\text{there is $\beta\in\DD$ with }\OO_{S,\alpha}=\OO_{S,\beta}\},
\]
where $\DD\not=\CC$ is a $\GL_2(\vk )$-equivalence class of special numbers.

Let $\DD\not=\CC$ be a $\GL_2(\vk )$-equivalence class such that $\CC (\DD )\not=\emptyset$.
We show by means of a local-to-global argument that $\CC (\DD )$ is the union of 
finitely many $G(\OO_S)$-equivalence classes. 

Fix $\alpha\in\CC (\DD )$, and then $\beta\in\DD$ with $\OO_{S,\alpha}=\OO_{S,\beta }$.
Let $T$ be the set of prime ideals $\fp$ of $\OO_S$ such that $\fp$ divides 
the discriminant ideal $\fd$ of $\OO_{S,\alpha}$,
or such that some prime ideal $\fP$ of $\OO_{S,L}$ above $\fp$
divides the ideal $\fa (\alpha ,\beta )$ of $\OO_{S,L}$
generated by the numbers
$\frac{\cross_{ijkl}(\beta )}{\cross_{ijkl}(\alpha )}-1$
for all pairwise distinct $i,j,k,l\in\{ 1\kdots n\}$. Clearly, $T$ is finite.
Next, choose $\alpha^*\in\CC (\DD )$ that we let vary, and then $\beta^*\in\DD$
with $\OO_{S,\alpha^*}=\OO_{S,\beta^*}$. 

Let $\fp$ be a prime ideal of $\OO_S$.
We apply the theory of Section \ref{section4} with $A=\OO_{\fp}$.
By \eqref{eq6.4} we have 
$\OO_{\fp ,\alpha}=\OO_{\fp ,\beta}$, $\OO_{\fp ,\alpha^*}=\OO_{\fp ,\beta^*}$.
Hence $(\alpha ,\beta )$ and $(\alpha^*,\beta^*)$ are two $\GL_2(\vk )$-equivalent special
pairs as in Proposition \ref{prop4.1}.
Let $C$ be the matrix from \eqref{eq4.1}, i.e., with $\alpha^*=C\alpha$, and put $\Delta :=\det C$.
We use that there is a finite set $\FF ([\Delta ])$ of $2\times 2$-matrices with entries
in $\OO_{\fp}$,
depending only on $\fp$ and on the ideal $[\Delta ]:=\Delta \OO_{\fp}$,  such 
that there is $U\in\GL_2(\OO_{\fp})$ with
\[
UC=:C_1\in \FF ([\Delta ]).
\]
Let $\alpha^{**}:=U\alpha^*=C_1\alpha$. 
Proposition \ref{prop4.1} implies that $[\Delta ]$ belongs to a finite set 
depending on $\alpha ,\beta$ and $\fp$,
hence so does $C_1$, and thus $\alpha^{**}$.
This implies that the $\GL_2(\OO_{\fp})$-equivalence class of $\alpha^*$  belongs to a finite 
collection depending on $\alpha ,\beta ,\fp$.

But for $\fp\not\in T$, i.e., for all but finitely many $\fp$, Proposition \ref{prop4.1} implies that $[\Delta ]=[1]$,
hence $\alpha^*$ is $\GL_2(\OO_{\fp})$-equivalent to $\alpha$.
Now from \eqref{eq6.2} it follows
that there is a finite collection of $G(\OO_S)$-equivalence classes
depending only on $\alpha ,\beta$ to which $\alpha^*$ must belong.
This shows that indeed, $\CC (\DD )$ is the union of finitely
many $G(\OO_S)$-equivalence classes, and completes step 2 of our proof of Theorem \ref{thm6.6}.
\end{proof}

\end{document}